%% file: small.tex
\documentclass{amsart}

\usepackage{amssymb,xspace}
\usepackage[all]{xy}
\usepackage{ifpdf}\ifpdf\else\usepackage[hypertex]{hyperref}\fi

\input{commands.tex}

\usepackage{hyperref}
\begin{document}

\begin{abstract} 
We compute the motivic Donaldson--Thomas theory of a small crepant resolution of a toric Calabi-Yau $3$-fold.
%the resolved conifold, in all chambers 
%of the space of stability conditions of the corresponding quiver. The answer is a product formula
%whose terms depend on the position of the stability vector, generalizing known 
%results for the corresponding numerical invariants. Our formulae imply in particular a motivic form of the 
%DT/PT correspondence for the resolved conifold. The answer for the motivic
%PT series is in full agreement with the prediction of the refined topological vertex formalism. 
\end{abstract}

\input{title.tex}
\setcounter{tocdepth}{1}
\tableofcontents
%\remfalse
\input{intro.tex}

\input{thanks.tex}

\input{NCCR.tex}
\input{MDT.tex}

\input{s3.tex}
\input{universal_small_2.tex}

\input{framing.tex}

\input{DTPT.tex}

\input{App.tex}

\bibliographystyle{amsalpha}
\bibliography{bib.bib}

%\input{biblio.tex}
%\bibliography{../tex/papers}
%\bibliographystyle{../tex/hamsplain}
\end{document}

%% file: commands.tex
\newtheorem{thm}{Theorem}[section]
\newtheorem{lem}[thm]{Lemma}
\newtheorem{prop}[thm]{Proposition}
\newtheorem{cor}[thm]{Corollary}

\newtheorem{ex}{Example}

\theoremstyle{definition}
\newtheorem{rk}[thm]{Remark}
\newtheorem*{rk*}{Remark}
\newtheorem{defn}[thm]{Definition}

\numberwithin{equation}{section}

\newcommand{\MC}{\lM_\C}
\newcommand{\tildeMC}{\wtl\lM_\C}
\newcommand{\C}{\mathbb{C}}

\newcommand{\MM}{\mathfrak{M}}
\newcommand{\R}{\mathbb{R}}
\newcommand{\LL}{\mathbb{L}}
\newcommand{\Jtilde}{\widetilde{J}} % framed algebra
\newcommand{\Qtilde}{\widetilde{Q}} % framed quiver
\newcommand{\That}{{\mathcal{T}}} % motivic torus, currently no hat on it

\newcommand{\h}{\frac{1}{2}}

\newcommand{\hind}{\tilde{I}}

\newcommand{\hI}{\hat{I}}
\newcommand{\tI}{\tilde{I}}

\newcommand{\Z}{\mathbb{Z}}
\newcommand{\tZ}{\tilde{\mathbb{Z}}}

\newcommand{\mr}[1]{{\mathrm{#1}}}

\newcommand{\mca}[1]{{\mathcal{#1}}}

\newif\ifrem\remtrue
\def\mat#1{\ensuremath{#1}\xspace}
\def\DMO{\DeclareMathOperator}
\def\set#1{\mat{\{#1\}}}
\def\sets#1#2{\mat{\{#1\mid#2\}}}

\def\wtl{\widetilde}

\def\cC{\mathbb{C}}

\def\cH{\mathbb{H}}
\def\cL{\mathbb{L}}
\def\cN{\mathbb{N}}

\def\cR{\mathbb{R}}
\def\cZ{\mathbb{Z}}
\def\lM{\mathcal{M}}

\def\al{\mat{\alpha}}
\def\be{\mat{\beta}}

\def\De{\mat{\Delta}}

\def\la{\mat{\lambda}}
\def\hi{\mat{\chi}}
\def\si{\mat{\sigma}}
\def\vi{\mat{\varphi}}
\def\ze{\mat{\zeta}}
\def\gM{\mathfrak{M}}

\DMO\Hom{Hom}
\DMO\End{End}
\DMO\GL{GL}
\DMO\Aut{Aut}
\DMO\Exp{Exp}
\DMO\Log{Log}
\DMO\Pow{Pow}
\DMO\Id{Id}

\def\crit{\operatorname{crit}}
\def\udim{\operatorname{\underline\dim}}
\def\vir{\mathrm{vir}}
\def\im{\mathrm{im}}
\def\re{\mathrm{re}}

\def\pser#1{[\![#1]\!]} %formal power series [[#1]]

\def\dd{\mat{\partial}}
\def\ms{\backslash} %minus set
\def\sb{\subset}

\def\n#1{\mat{\lvert#1\rvert}}

\def\inv{^{-1}}

\def\oh{\mat{\frac12}}
\def\ie{i.e.\ }

\def\GG{G}

\newcommand{\inj}{\hookrightarrow}
\newcommand{\surj}{\twoheadrightarrow}

\renewcommand{\phi}{{\varphi}}

\newcommand{\half}{\frac{1}{2}}

%% file: title.tex
\title[Motivic Donaldson--Thomas invariants of toric small crepant resolutions]{Motivic Donaldson--Thomas invariants of toric small crepant resolutions
% and the refined topological vertex}
}

\author[Morrison]{Andrew Morrison}
\author[Nagao]{Kentaro Nagao}
%\address{Graduate School of Mathematics, Nagoya University}
%\email{kentaron@math.nagoya-u.ac.jp}

%\begin{abstract}
%\end{abstract}

\maketitle

%% file: intro.tex
\thispagestyle{empty}

\section*{Introduction}

This paper is a continuation of \cite{mmns}. 
We study the {\it motivic Donaldson--Thomas invariants} 
of non-commutative and commutative crepant resolutions of the affine toric Calabi--Yau $3$-fold $\{XY-Z^{N_0}W^{N_1}\}\subset \mathbb{C}^4$.

\smallskip

A {\it Donaldson--Thomas} (DT) {\it invariant} of a Calabi--Yau $3$-fold $Y$ is a counting invariant of 
coherent sheaves on $Y$, introduced in \cite{thomas-dt} as a holomorphic analogue of the Casson invariant 
of a real $3$-manifold. A component of the moduli space of stable coherent sheaves on $Y$ carries 
a symmetric obstruction theory and a virtual fundamental cycle \cite{behrend-fantechi-intrinsic,behrend-fantechi}. 
A DT invariant of a compact $Y$ is then defined as the integral of the constant function $1$ over the virtual 
fundamental cycle of the moduli space.

It is known that the moduli space of coherent sheaves on $Y$ can be locally described as the critical locus of a 
function, the {\it holomorphic Chern--Simons functional} (see \cite{joyce-song}). Behrend provided 
a description of DT invariants in terms of the Euler characteristic of the {\it Milnor fiber} of the CS 
functional~\cite{behrend-dt}. Inspired by this result, the proposal of \cite{ks,behrend_bryan_szendroi} was  
to study the {\it motivic Milnor fiber} of the CS functional as a motivic refinement of the DT invariant. Such a refinement 
had been expected in string theory \cite{RTV,dimofte-gukov}.

\smallskip

On the other hand, in~\cite{szendroi-ncdt}, it was proposed to study counting invariants for the non-commutative crepant resolution (NCCR) of the conifold, which are called non-commutative Donaldson--Thomas (ncDT) invariants.
It was also conjectured there that ncDT and DT invariants are related by wall crossing.
The paper~\cite{nagao-nakajima} realized this, by
\begin{itemize}
\item describing the chamber structure on the space of stability parameters for the NCCR,
\item finding chambers which correspond to geometric DT and stable pair (PT), as well as ncDT invariants, and
\item computing the generating function of DT type invariants for each chamber.
\end{itemize}
%In \cite{mmns}, we study the motivic Donaldson-Thomas invariant of the conifold following the idea of Szendr\H oi~\cite{szendroi_non-commutative} and Nagao and Nakajima~\cite{nagao_counting}.
For the conifold, the dimension of the fiber of the crepant resolution is less than $2$ (we say that the resolution is small).
This condition plays an important role in many places of the paper. 
Affine toric Calabi--Yau $3$-folds which have small crepant resolutions are classified as follows:
\begin{enumerate}
\item $\mathcal{X}=\mathcal{X}_{N_0,N_1}:=
\{XY-Z^{N_0}W^{N_1}\}$ for $N_0>0$ and $N_1\geq 0$, or 
\item $\mathcal{X}=\mathcal{X}_{(\Z/2\Z)^2}:=\C^3/(\Z/2\Z)^2$ where $(\Z/2\Z)^2$ acts on $\C^3$ with weights $(1,0)$, $(0,1)$ and $(1,1)$.
\end{enumerate}
\begin{figure}[htbp]
  \centering
  \input{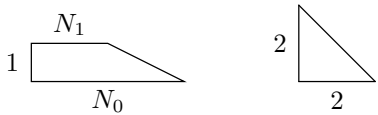}
  \caption{Polygons for $\mathcal{X}_{N_0,N_1}$ and $\mathcal{X}_{(\Z/2\Z)^2}$}
\label{fig1}
\end{figure}

In \cite{3tcy}, counting invariants for non-commutative and commutative crepant resolutions of $\{XY-Z^{N_0}W^{N_1}\}$ were studied. First, we provided descriptions of NCCRs of $\{XY-Z^{N_0}W^{N_1}\}$ in terms of a quiver with potential. Given $N_0$ and $N_1$, the quivers with potential are not unique. However it was also shown that any such quivers with potential are related by a sequence of mutations. Finally, generalizations of the results in \cite{nagao-nakajima} are given.

\smallskip

In \cite{mmns}, we provided motivic refinements of formulae in \cite{nagao-nakajima}.
For the proof, we needed one explicit evaluation of the ``universal'' series (\cite[\S 2]{mmns}) and  
a wall-crossing argument (\cite[\S 3]{mmns}).

\smallskip

In this paper, we will show similar formulae for $\{XY-Z^{N_0}W^{N_1}\}$, that is, motivic refinements of the formulae in \cite{3tcy}. 
The wall-crossing argument works without modifications (\S \ref{sec_framing}, \ref{sec_DTPT}), while the evaluation part is more involved (Theorem \ref{thm_A}). 
Our strategy is as follows:
\begin{itemize}
\item
First, in \S \ref{sec_univ1}, we evaluate the universal series for a specific NCCR using a generalization of the calculation \cite[\S 2.2]{mmns}. 
\item 
Then, in \S \ref{sec_univ2}, we evaluate the universal series for a general NCCR. In \cite{motivic_WC}, the second author provided a formula which describes how the universal series changes under mutation (\S \ref{sec_framing}, \ref{sec_DTPT}). 
Although we assume that the quiver has no loops and $2$-cycles in \cite{motivic_WC}, we can apply a parallel argument in our setting as well.
\end{itemize}
Since any two NCCRs are related by a sequence of mutations, the evaluation is done.

\section*{Main result}
Let $\Gamma$ be the quadrilateral (or the triangle in case $N_1=0$) as in Figure \ref{fig1} and $\sigma$ be a partition $\Gamma$, that is, a division of $\Gamma$ into $N$-tuples of triangles with area $1/2$.
We will associate $\sigma$ with a quiver with superpotential $(Q_{\sigma},\omega_{\sigma})$.
The set of vertices of the quiver $Q_{\sigma}$ is $\hI:=\mathbb{Z}/N\mathbb{Z}$, which is identified with $\{0,\ldots,N-1\}$. 
A vertex has a loop if and only if it is in the subset $\hI_r\subset \hI$ (see \eqref{eq_Ir} for the definition).
It is shown in \cite[\S 1]{3tcy} that the Jacobian algebra $J_\sigma:=J{(Q_{\sigma},\omega_{\sigma})}$ is an NCCR of 
\[\mathcal{X}:=\mathrm{Spec}\left(\C[X,Y,Z,W] /(XY-Z^{N_0}W^{N_1})\right).\]

Let $\Delta$ be the set of roots of type $\hat{A}_{N}$ and $\Delta_{\sigma,+}$ (resp. $\Delta^{\re}_{\sigma,+}$, $\Delta^{\im}_{\sigma,+}$) denote the set of positive (resp. positive real, positive imaginary) roots. 
\footnote{From the view point of the root system, a choice of a partition $\sigma$ corresponds to a choice of a set of simple roots.}
%To each element $\al\in \Delta_+$, we associate a finite product as follows: 

\smallskip

For $\al\in\cN^{\hI}$, let $\gM(J_\sigma,\al)$ be the moduli stack of $J_\sigma$-modules ${V}$ with $\udim{V}=\al$.
We define the generating series of the motivic DT invariants of $(Q_\sigma,W_\sigma)$ by
\[
A_U^\sigma(y)=
A_U^\sigma(y_0,\ldots,y_{N-1})
:=\sum_{\al\in\cN^{Q_0}}[\gM(J_\sigma,\al)]_\vir \cdot y^\al
%=\sum_{\al\in\cN^{Q_0}}
%\frac{[\crit(f_\al)]_\vir}{[\GG_\al]_\vir}\cdot y^\al
%\in\That_Q.
\in \MC[[y_0,\ldots,y_{N-1}]].\footnote{For the wall-crossing of motivic DT theory, a twisted product on $y_\alpha$'s twisted by the Euler form plays a crucial role. In this case, the twisted product coincides with the usual commutative product since the Euler form is trivial.}
\]
Here $y^\al:=\prod (y_i)^{\alpha_i}$ and $[\bullet]_\vir$ denotes the {\it virtual motive} (see Section \ref{subsec_11}), an element of a suitable ring of motives $\MC$. 
The subscript referring to the fact that we think of this series as the universal series.

%Let $A_U^\sigma:=\sum_{\alpha}A^\sigma(\alpha)\cdot y^\alpha$ be the universla Donaldson-Thomas series for $(Q_{\sigma},\omega_{\sigma})$.
To each root $\alpha\in \Delta_{\sigma,+}$, we associate an infinite product as follows:
\begin{itemize}
\item
for a real root $\al\in \Delta^{\re}_{\sigma,+}$ such that $\sum_{k\notin \hat{I}_r}\alpha_k$ is odd, put
\begin{align*}
A^\al(y) &:= 
\mathrm{Exp}\left(\frac{-\mathbb{L}^{-1/2}}{1-\mathbb{L}^{-1}}y^\alpha\right)\\
& =
\prod_{j\geq 0}\left(1-\mathbb{L}^{-j-1/2}y^\alpha\right) 
\end{align*}
\item
for a real root $\al\in \Delta^{\re}_{\sigma,+}$ such that $\sum_{k\notin \hat{I}_r}\alpha_k$ is even, put
\begin{align*}
A^\al(y) &:= 
\mathrm{Exp}\left(\frac{1}{1-\mathbb{L}^{-1}}y^\alpha\right)\\
& =
\prod_{j\geq 0}
\left(1-\mathbb{L}^{-j}y^\alpha\right)^{-1} 
\end{align*}
\item
for an imaginary root $\al\in\Delta^{\im}_{\sigma,+}$, put
\begin{align*}
A^\al(y) &:= 
\mathrm{Exp}\left(\frac{N-1+\mathbb{L}}{1-\mathbb{L}^{-1}}y^\alpha\right)\\
& =
\prod_{j\geq 0}
\left(1-\mathbb{L}^{-j}y^\alpha\right)^{1-N} 
\cdot 
\left(1-\mathbb{L}^{-j+1}y^\alpha\right)^{-1}. 
\end{align*}
\end{itemize}
The main result of this paper is the following formula:
\begin{thm}\label{thm_A}
%\label{thm main universal result}
\[
A^\sigma_U(y)=
\prod_{\al\in\De_{\sigma,+}}
A^{\al}(y).
\]
%\label{eq univ exp form}
\end{thm}
This is proved in \S \ref{sec_univ1} and \S \ref{subsec_univ2}.

\section*{Corollaries}

Let $\Jtilde_\sigma=J(\Qtilde_\sigma,W_\sigma)$ be 
the framed algebra given by adding the new vertex $\infty$ and the new arrow from $\infty$ to $0$ to the quiver of $J_\sigma$.
In \cite{nagao-nakajima}, the authors introduce a notion of $\ze$-(semi)stability of 
$\Jtilde$-modules $\wtl{V}$ with $\dim \wtl{V}_\infty\leq 1$ for a stability 
parameter $\ze\in \mathbb{R}^{\hI}$.

For $\al\in\cN^{\hI}$, let ${\MM}_\ze(\Jtilde,\al)$ be the moduli space of $\ze$-stable $\Jtilde$-modules $\wtl{V}$ with $\udim\wtl{V}=(\al,1)$.
We want to compute the motivic generating series
\[
Z_\ze(y)=
Z_\ze(y_0,\ldots,y_{N-1}):=
\sum_{\al\in\cN^{\hI}}
\Bigl[{\MM}_\ze\bigl(\Jtilde,\al\bigr)\Bigr]_\vir \cdot y^\alpha\in \MC[[y_0,\ldots,y_{N-1}]].
\]
%Here $[\bullet]_\vir$ denotes the {\it virtual motive} (see Section \ref{subsec_11}), an element of a suitable ring of motives $\MC$. 

%As proved in~\cite{nagao_counting}, the stability parameter space $\R^2$ is a countable union of chambers, within which the moduli spaces and therefore the generating series $Z_\ze$ remain unchanged. The chambers are separated by a set of walls, which is defined in terms of the root system of type $\hat{A}_N$. 
To each root $\al\in \Delta_{\sigma,+}$, we put
\[
Z_\alpha(y_0,\ldots,y_{N-1}):=\frac{A^\al(-\mathbb{L}^{1/2}y_0,y_1,\ldots,y_{N-1})}{A^\al(-\mathbb{L}^{-1/2}y_0,y_1,\ldots,y_{N-1})}.
\]
They are given as follows:
\begin{itemize}
\item
for a real root $\al\in \Delta^{\re}_{\sigma,+}$ such that $\sum_{k\notin \hat{I}_r}\alpha_k$ is odd, we have
\[
Z_\al(-y_0,\ldots,y_{N-1})
=
\prod_{i=0}^{\alpha_0-1} \left(1-\mathbb{L}^{-\frac{\alpha_0}{2}+\frac{1}{2}+i}y^\alpha\right),
\]
%\begin{align*}
%A^\al(\mathbf{y}) &:= 
%\mathrm{Exp}\left(\frac{-\mathbb{L}^{-1/2}}{1-\mathbb{L}^{-1}}y^\alpha\right)\\
%& =
%\prod_{j\geq 0}\left(1-\mathbb{L}^{-j-1/2}y^\alpha\right) 
%\end{align*}
\item
for a real root $\al\in \Delta^{\re}_{\sigma,+}$ such that $\sum_{k\notin \hat{I}_r}\alpha_k$ is even, we have
\[
Z_\al(-y_0,\ldots,y_{N-1})
=
\prod_{i=0}^{\alpha_0-1} 
\left(1-\mathbb{L}^{-\frac{\alpha_0}{2}+1+i}y^\alpha\right)^{-1},
\]
\item
for an imaginary root $\al\in\Delta^{\im}_{\sigma,+}$, we have
\[
Z_\al(-y_0,\ldots,y_{N-1})=\prod_{i=0}^{\alpha_0-1} \left(1-\mathbb{L}^{-\frac{\alpha_0}{2}+1+i}y^\alpha\right)^{1-N} \cdot \left(1-\mathbb{L}^{-\frac{\alpha_0}{2}+2 +j}y^\alpha\right)^{-1}
\]

%\[Z_\al(\mathbf{y})=\left(1-\mathbb{L}^{-\frac{\alpha_0}{2}+1}y^\alpha\right)^{1-N}\cdot \left(\,\prod_{i=2}^{\alpha_0} \left(1-\mathbb{L}^{-\frac{\alpha_0}{2}+i}y^\alpha\right)^{-N}\right)\cdot \left(1-\mathbb{L}^{\frac{\alpha_0}{2}+1}y^\alpha\right)^{-1}\]
%\begin{align*}
%A^\al(\mathbf{y}) &:= 
%\mathrm{Exp}\left(\frac{N-1+\mathbb{L}}{1-\mathbb{L}^{-1}}y^\alpha\right)\\
%& =
%\prod_{j\geq 0}
%\left(1-\mathbb{L}^{-j}y^\alpha\right)^{1=n} 
%\left(1-\mathbb{L}^{-j+1}y^\alpha\right)^{-1} 
%\cdot 
%\left(1-\mathbb{L}^{-j+1}y^\alpha\right)^{-1}. 
%\end{align*}
\end{itemize}
Applying the same argument as \cite[\S 3]{mmns}, we get the following formula (\S \ref{sec_framing}):
%{\def\thethm{1}
\begin{cor}\label{cor_02}
%\label{announce_thm_main_res}
For $\ze\in\R^{\tI}$ not orthogonal to any root, we have
\[
Z_\ze(y)=
\prod_{\substack{\al\in\De_{\sigma,+}\\ \ze\cdot\al<0}}
Z_{\al}(y_0,\ldots,y_{N-1}).
\]
\end{cor}
%}
By~\cite{behrend-dt, behrend_bryan_szendroi}, the specialization $Z_\ze(y)|_{\LL^\oh\to 1}$ is the DT-type series at the generic 
stability parameter $\ze$, computed in ~\cite{3tcy}.

\smallskip

Let $\mathcal{Y}_\sigma\to \mathcal{X}$ be the crepant resolution corresponding to $\sigma$.  
The noncommutative crepant resolution $J_{\sigma}$ is derived equivalent to $\mathcal{Y}_\sigma$.
In \cite[\S 3]{nagao-nakajima}, we find a stability parameter $\zeta_{\mr{DT}}$ (resp. $\zeta_{\mr{PT}}$) such that the moduli space coincides with the Hilbert scheme (resp. the stable pair moduli space) for $\mathcal{Y}_\sigma$.

Let $Z^\sigma_{\rm DT}(s,T_1,\ldots,T_{N-1})$ (resp. $Z^\sigma_{\rm DT}(s,T_1,\ldots,T_{N-1})$) be the generating function of DT (resp. PT) invariants of $\mathcal{Y}_\sigma$. 
Here $s$ is the variable for the homology class of a point and 
$T_i$ is the variable for the homology class of the $i$-th component $C_i$ of the exceptional curve.
The variable change induced by the derived equivalence is given as follows:
\[
s:=y_0\cdot y_1\cdot\cdots\cdot y_{N-1},\quad 
%t_{[a,b]}=y_1\cdot\cdots\cdot y_{b}
T_i=y_i.
\]

For $1\leq a\leq b\leq N-1$, we put
\[
C_{[a,b]}:=[C_a]+\cdots +[C_b] \in H_2(\mathcal{Y}_\sigma,\mathbb{Z}),
\]
where $C_i$ is a component of the exceptional curve and let 
\[
T_{[a,b]}=T_a\cdot\cdots\cdot T_b
\]
be the corresponding monomial.
Let $c(a,b)$ denote the number of $(-1,-1)$-curves in $\{C_i\mid a\leq i\leq b\}$.
%We set 
%\[
%q:=y_0\cdot y_1\cdot\cdots\cdot y_{N-1},\quad 
%q_{[a,b]}=y_1\cdot\cdots\cdot y_{b}
%\]
We define infinite products as follows:
%To each root $\al\in \Delta$, we put
%\[
%Z_\alpha(\mathbf{y}):=\frac{A^\al(\mathbb{L}^{1/2}y_0,y_1,\ldots,y_{N-1})}{A^\al(\mathbb{L}^{-1/2}y_0,y_1,\ldots,y_{N-1})}.
%\]
%They are given as follows:
\begin{itemize}
\item 
If $c(a,b)$ is odd, we put
\[
Z_{[a,b]}=Z_{[a,b]}(s,T_{[a,b]}):=
\prod_{n=1}^\infty\left(\prod_{i=0}^{n-1} \left(1-\mathbb{L}^{-\frac{n}{2}+\frac{1}{2}+i}\cdot (-s)^n\cdot T_{[a,b]}\right)\right).
\]
\item 
If $c(a,b)$ is even, we put
\[
Z_{[a,b]}=Z_{[a,b]}(s,T_{[a,b]}):=
\prod_{n=1}^\infty\left(\prod_{i=0}^{n-1} \left(1-\mathbb{L}^{-\frac{n}{2}+1+i}\cdot (-s)^n\cdot T_{[a,b]}\right)^{-1}\right).
\]
\item For imaginary roots, we put
\[ Z_\mathrm{im} = Z_\mathrm{im}(s) := \prod_{n=1}^\infty \left( \prod_{i=0}^{n-1}\left(1-\mathbb{L}^{-\frac{n}{2}+1+i}(-s)^n\right)^{1-N}\left(1-\mathbb{L}^{-\frac{n}{2}+2+i}(-s)^n\right)^{-1} \right).\]
%\[Z_\mathrm{im}=Z_\mathrm{im}(s):=\prod_{n=1}^\infty\left(\left(1-\mathbb{L}^{-\frac{n}{2}+1}s^n\right)^{1-N}\cdot \left(\,\prod_{i=2}^{n} \left(1-\mathbb{L}^{-\frac{n}{2}+i}s^n\right)^{-N}\right)\cdot \left(1-\mathbb{L}^{\frac{n}{2}+1}s^n\right)^{-1}\right)\]
\end{itemize}
%{\def\thethm{2}
\begin{cor}\label{announce_thm_main_cor}
\begin{itemize}
\item[(1)]
The refined DT and PT series of $\mathcal{Y}_\sigma$ are given by the formulae :
\[
Z_{\rm DT}(s,T_1,\ldots,T_{N-1}) = 
Z_\mathrm{im}(s) \cdot \prod_{1\leq a\leq b\leq N-1}Z_{[a,b]}(s,T_{[a,b]})
\]
and
\[
Z_{\rm PT}(s,T_1,\ldots,T_{N-1}) =
\prod_{1\leq a\leq b\leq N-1}Z_{[a,b]}(s,T_{[a,b]})
\]
%written in the geometric variables $s,T$, with $s$ representing the point class and $T$ representing the curve class as usual.
\item[(2)]
The generating function of virtual motives of the Hilbert scheme of points on $\mathcal{Y}_\sigma$ is given by the formula :
\[
Z_\textup{$0$-dim}(s):=
\sum_{n=0}^\infty \left[(\mathcal{Y}_\sigma)^{[n]}\right]_\mr{vir}\cdot s^n
=Z_\mathrm{im}.
\]
\item[(3)]The refined version of DT-PT correspondence for $\mathcal{Y}_\sigma$ holds :
\[
Z_{\rm DT}(s,T_1,\ldots,T_{N-1})=
%\left(\sum_{n=0}^\infty \left[(\mathcal{Y}_\sigma)^{[n]}\right]_\mr{vir}\cdot s^n\right)\cdot
Z_\textup{$0$-dim}(s)\times
Z_{\rm PT}(s,T_1,\ldots,T_{N-1}).
\]
\end{itemize}
\end{cor}
%}
\begin{rk*}
%\begin{itemize}
%\item[(A)]
The formula in (2) is a direct consequence of the formula for $Z_{\rm DT}$ in (1), since the polynomial in the $T_{[a,b]}$ variables does not contribute.
%\item[(B)]
%The formula of $Z_{\rm PT}$ in (1) coincides with  
\end{rk*}

%% file: thanks.tex
\section*{Acknowledgements}
The first author thanks his Ph.D. supervisor Jim Bryan for his mathematical support over the last years. The authors thank Sergey Mozgovoy for  valuable discussion. A.M. is partially supported by a Four Year Doctoral Fellowship (4YF), University of British Columbia. K.N.\ is supported by the Grant-in-Aid for Research Activity Start-up (No.\ 22840023) and for 
Scientific Research (S) (No.\ 22224001).

%% file: NCCR.tex
\section{Root system of type $\hat{A}_N$}
Let $N_0>0$ and $N_1\geq 0$ be integers such that $N_0\geq N_1$ and set $N=N_0+N_1$.
We set
\begin{align*}
I&=\left\{1,\ldots,N-1\right\},\\
\hat{I}&=\left\{0,1,\ldots,N-1\right\},\\
\tilde{I}&=\left\{\h,\frac{3}{2},\ldots,N-\h\right\},\\
\tilde{\Z}&=\left\{n+\h\,\Big|\, n\in\Z\right\}.
\end{align*}
For $l\in\Z$ and $j\in\tZ$, let $\underline{l}\in\hI$ and $\underline{j}\in\tI$ be the elements 
such that $l-\underline{l}\equiv j-\underline{j}\equiv 0$ modulo $N$.

Let $\mathbb{Z}^{\hat{I}}$ be the free Abelian group with basis $\{\alpha_i\mid i\in \hI\}$, where $\alpha_i$ is called a simple root.
We put 
\begin{align*}
%\begin{align*}
\Delta^{\mr{fin}}_+  &:= \{ \alpha_{[a,b]}:=\alpha_a+\cdots\alpha_b\mid 1\leq a\leq b\leq N-1\}\\
%,\quad
%\Lambda^{\mr{fin}}_-  := -\Lambda^{\mr{fin},+}.
%\]
%and $\Lambda^{\mr{fin}}:=\Lambda^{\mr{fin}}_+\sqcup \Lambda^{\mr{fin}}_-$.
%and
%\begin{align*}
\Delta_+^{\mr{re},+}  &:= 
\{ \alpha_{[a,b]}+n\cdot \delta 
\mid 
\alpha_{[a,b]}\in \Delta^{\mr{fin}}_+, 
n\in \mathbb{Z}_{\geq 0} \}\\
\Delta_+^{\mr{re},-}  &:= 
\{ -\alpha_{[a,b]}+n\cdot \delta 
\mid 
\alpha_{[a,b]}\in \Delta^{\mr{fin}}_+, 
n\in \mathbb{Z}_{> 0} \}
\end{align*}
and
\[
\Delta_+^{\mr{re}}  := \Delta_+^{\mr{re},+} \sqcup \Delta_+^{\mr{re},-},\quad
\Delta^{\mr{im}}_+  := \{ n\cdot \delta \mid  n\in \mathbb{Z}_{> 0}\}
\]
where $\delta:=\alpha_0+\cdots+\alpha_{N-1}$ is the (positive minimal) imaginary root.

For $k\in \hI$, the simple reflection at $k$ is the group homomorphism given by
\[
\begin{array}{ccc}
\mathbb{Z}^{\hI} & \to & \mathbb{Z}^{\hI}\\
\alpha_i & \mapsto & \alpha_i-C_{ik}\cdot \alpha_k
\end{array}
\]
where $C$ is the Cartan matrix of type $\hat{A}_N$.
This gives a self-bijection of $\Delta^{\mr{re},+}_+\backslash \{\alpha_k\}$.

\section{Noncommutative crepant resolutions}\label{sec_NCCR}
\subsection{Quivers with potential}
%We denote by $\Gamma$ the quadrilateral (or the triangle in case $N_1=0$) in $(\R^2)_{x,y}=\{(x,y)\}$ with vertices $(0,0)$, $(0,1)$, $(N_0,0)$ and $(N_1,1)$.
%Let $M^\vee\simeq \Z^3$ be the lattice with basis $\{x^\vee,y^\vee,z^\vee\}$, 
We denote by $\Gamma$ the quadrilateral (or the triangle in case $N_1=0$) with vertices $(0,0)$, $(0,1)$, $(N_0,0)$ and $(N_1,1)$.
Note that the affine toric Calabi--Yau $3$-fold corresponding to $\Gamma$ is $\mathcal{X} = \{XY-Z^{N_0}W^{N_1}\}$.

A partition $\sigma$ of $\Gamma$ is a pair of functions $\sigma_x\colon\tI\to\tZ$ and $\sigma_y\colon\tI\to\{0,1\}$ such that 
\begin{itemize}
\item $\sigma(i):=(\sigma_x(i),\sigma_y(i))$ gives a bijection between $\tI$ and the following set:
%a permutation of the set
\[
\left\{\left(\h,0\right),\left(\frac{3}{2},0\right),\ldots,\left(N_0-\h,0\right),\left(\h,1\right),\left(\frac{3}{2},1\right),\ldots,\left(N_1-\h,1\right)\right\},
\]  
\item if $i<j$ and $\sigma_y(i)=\sigma_y(j)$ then $\sigma_x(i)>\sigma_x(j)$.
\end{itemize}
%sequence $(p_1,\ldots, p_{N_0+N_1})$ such that $p_i=0$ or $1$ for any $1\leq i\leq N_0+N_1$ and $\sum_{i=1}^{N_0+N_1}p_i=N_1$. 
Giving a partition $\sigma$ of $\Gamma$ is equivalent to dividing $\Gamma$ into $N$-tuples of triangles $\{T_i\}_{i\in\tI}$ with area $1/2$ so that $T_i$ has $(\sigma_x(i)\pm 1/2, \sigma_y(i))$ as its vertices.
Let $\Gamma_\sigma$ be the corresponding diagram, $\Delta_\sigma$ be the fan and $f_\sigma\colon \mca{Y}_\sigma\to \mca{X}$ be the crepant resolution of $\mca{X}$. 
We put
\begin{equation}\label{eq_Ir}
\hI _r:=\left\{k\in \hI\,\Big|\,\sigma_y(k-\h)=\sigma_y(k+\h)\right\}.
\end{equation}

\begin{ex}\label{example}
Let us consider as an example the case $N_0=4$, $N_1=2$ and 
\[
(\sigma(i))_{i\in \tI}=\left(\left(\frac{7}{2},0\right),\left(\frac{3}{2},1\right),\left(\frac{5}{2},0\right),\left(\frac{3}{2},0\right),\left(\frac{1}{2},1\right), \left(\frac{1}{2},0\right)\right).
\]
We show the corresponding diagram $\Gamma_\sigma$ in Figure \ref{fig:Q}.
\begin{figure}[htbp]
  \centering
  \input{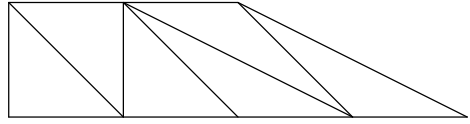}
  \caption{$\Gamma_\sigma$}
  \label{fig:Q}
\end{figure}
\end{ex}
Let $S$ be the union of an infinite number of rhombi with edge length $1$ as in Figure \ref{fig:S} which is located so that the centers of the rhombi are on a line parallel to the $x$-axis in $\R^2$  and $H$ be the union of infinite number of hexagons with edge length $1$ as in Figure \ref{fig:H} which is located so that the centers of the hexagons are in a line parallel to the $x$-axis in $\R^2$.
\begin{figure}[htbp]
  \centering
  \input{4gon-2.tpc}
  \caption{S}
  \label{fig:S}
\end{figure}
\begin{figure}[htbp]
  \centering
  \input{6gon.tpc}
  \caption{H}
  \label{fig:H}
\end{figure}
We make the sequence $\tau=\tau_\sigma\colon \Z\to \{S,H\}$ which maps $l$ to $S$ (resp. $H$) if $l$ modulo $N$ is not in $\hI _r$ (resp. is in $\hI _r$) and cover the whole plane $\R^2$ by arranging $S$'s and $H$'s according to this sequence (see Figure \ref{fig:P}). 
We regard this as a graph on the $2$-dimensional torus $\R^2/\Lambda$, where $\Lambda$ is the lattice generated by $(\sqrt{3},0)$ and $(N_0-N_1,(N_0-N_1)\sqrt{3}+N_1)$.
\begin{figure}[htbp]
  \centering
  \input{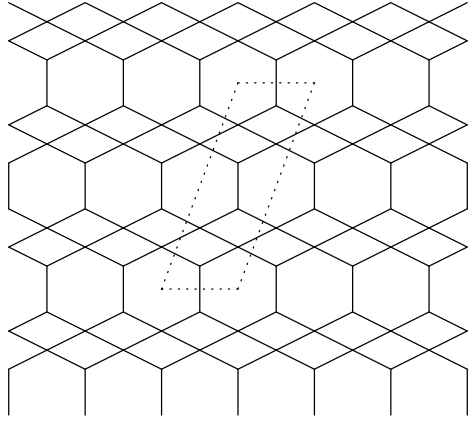}
  \caption{$P_\sigma$ in case Example \ref{example}}
  \label{fig:P}
\end{figure}
We can color the vertices of this graph black or white so that each edge connects a black vertex and a white one. 
Let $P_{\sigma}$ denote this bipartite graph on the torus. 
For each edge $h^\vee$ in $P_{\sigma}$, we make its dual edge $h$ directed so that we see the black end of $h^\vee$ on our right hand side when we cross $h^\vee$ along $h$ in the given direction. 
%The resulting quiver coincides with $Q_{\sigma}$. 
Let $Q_\sigma$ denote the resulting quiver.
The set of vertices of the quiver $Q_{\sigma}$ is $\hI$, which is identified with $\Z/N\Z$. 
The set of edges of the quiver $Q_{\sigma}$ is given by
\[
H:=\left(\coprod_{i\in\hind}h^+_i\right)\sqcup\left(\coprod_{i\in\hind}h^-_i\right)\sqcup\left(\coprod_{k\in \hI _r}r_k\right).
\]
Here $h^+_i$ (resp. $h^-_i$) is an edge from $i-\h$ to $i+\h$ (resp. from $i+\h$ to $i-\h$), $r_k$ is an edge from $k$ to itself.

For each vertex $q$ of $P_{\sigma}$, let $\omega_q$ be the potential\footnote{A potential of a quiver $Q$ is an element in $\C Q/[\C Q,\C Q]$, i.e. a linear combination of equivalence classes of cyclic paths in $Q$ where two paths are equivalent if they coincide after a cyclic rotation.} which is the composition of all arrows in $Q_{\sigma}$ corresponding to edges in $P_{\sigma}$ with $q$ as their ends.
We define 
\[
\omega_\sigma:=\sum_{\text{$q$ : black}}\omega_q-\sum_{\text{$q$ : white}}\omega_q.
\]
The relations of the Jacobian algebra are as follows:
\begin{itemize}
\item $h^+_i\circ r_{i-\h}=r_{i+\h}\circ h^+_i$ and $r_{i-\h}\circ h^-_i=h^-_i\circ r_{i+\h}$ for $i\in \hind$ such that $i-\h$, $i+\h\in \hI _r$.
\item $h^+_i\circ r_{i-\h}=h^-_{i+1}\circ h^+_{i+1}\circ h^+_i$ and $r_{i-\h}\circ h^-_i=h^-_i\circ h^-_{i+1}\circ h^+_{i+1}$ for $i\in \hind$ such that $i-\h\in \hI _r$, $i+\h\notin \hI _r$.
\item $h^+_i\circ h^+_{i-1}\circ h^-_{i-1}=r_{i+\h}\circ h^+_i$ and $h^+_{i-1}\circ h^-_{i-1}\circ h^-_i=h^-_i\circ r_{i+\h}$ for $i\in \hind$ such that $i-\h\notin \hI _r$, $i+\h\in \hI _r$.
\item $h^+_i\circ h^+_{i-1}\circ h^-_{i-1}=h^-_{i+1}\circ h^+_{i+1}\circ h^+_i$ and $h^+_{i-1}\circ h^-_{i-1}\circ h^-_i=h^-_i\circ h^-_{i+1}\circ h^+_{i+1}$ for $i\in \hind$ such that $i-\h$, $i+\h\notin \hI _r$.
\item $h^+_{i-\h}\circ h^-_{i-\h}=h^-_{i+\h}\circ h^+_{i+\h}$ for $k\in \hI _r$.
\end{itemize}

\subsection{NCCR and derived equivalence}
Let $\pi\colon \mathcal{Y}_\sigma\to \mathcal{X}$ be the crepant resolution corresponding to $\sigma$.
\begin{thm}\cite[Theorem 1.15 and Theorem 1.20]{3tcy}
\[
D^b(\mathrm{mod}J_\sigma)
\simeq
D^b(\mathrm{Coh}\mathcal{Y}_\sigma)
\]
\end{thm}
The equivalence is given by an explicit tilting vector bundle which is a direct sum of line bundles \cite[Theorem 1.10]{3tcy}.
In particular, the following map is compatible with the derived equivalence:
%The exeptional fiber $\pi^{-1}(0)\subset Y_\sigma$ is an union of $\mathbb{P}^1$'s  $C_1,\ldots, C_{N-1}$. 
%The homology classes $[C_1],\ldots, [C_{N-1}]\in H^2(Y_\sigma,\mathbb{Z})$ forms a basis. 
%The map 
\[
\begin{array}{ccccc}
H^0(Y_\sigma,\mathbb{Z})
&
\oplus 
&
H^2(Y_\sigma,\mathbb{Z})
&
\to
&
\mathbb{Z}^I\\
{[}\mathrm{pt}{]}
&
&
&\mapsto
&\delta
\\
&
&[C_i]
&\mapsto
&\alpha_i
%\mathbb{Z}^I & \to & 
%H^0(Y_\sigma,\mathbb{Z})
% \oplus 
%H^2(Y_\sigma,\mathbb{Z})\\
%(v_i)
%& \mapsto &
%(v_0, (v_i-v_0))
\end{array}
\]
where $\alpha_i$ is the $i$-th fundamental vector and $\delta:=\alpha_0+\alpha_1+\cdots \alpha_{N-1}$.

\subsection{Mutation and derived equivalence}\label{subsec_mutation}
Derksen--Weyman--Zelevinsky's mutation (\cite{DWZ2}) of a quiver with a potential induces a derived equivalence of the derived categories of Ginzburg's dgas (\cite{dong-keller}). 
Moreover, the relation between the module categories of Jacobian algebras has a description in terms of torsion pair and tilting, which plays a crucial role for the wall-crossing formulae (\cite{ks, cluster-via-DT}). 
In this paper, we can not apply \cite{DWZ2} and \cite{dong-keller} since we have loops and oriented $2$-cycles in the quiver. 
In this subsection, we see derived equivalences and descriptions of module categories using the explicit computations given in \cite[\S 3]{3tcy}.

Let $k$ be an edge of the partition $\sigma$ which is a diagonal of a parallelogram. 
Note that such $k$ corresponds to a vertex without loops. 
Let $\sigma'$ denote the partition which is obtained by a ``flip'' of the edge $k$.

Let $P_i$ be the indecomposable projective $J_\sigma$-module associated to a vertex $i$. 
Note that as a vector space $P_i$ is the space of linear combinations of path ending at the vertex $i$.
We define 
%The equivalence is given by the tilting object $\oplus P'_i$ where $P'_i=P_i$ for $i\neq k$ and 
\[
P'_k:=\mathrm{coker}(P_k\to P_{k-1}\oplus P_{k+1}).
\]
and put $P'_i=P_i$ for $i\neq k$. 
Here the map $P_k\to P_{k\pm 1}$ above is induced by the arrow from $k$ to $k\pm 1$.
\begin{thm}\cite[Proposition 3.1]{3tcy}\label{thm_mutation}
\begin{itemize}
\item[(1)]
%\begin{equation}\label{eq_mutation}
\[
\mathrm{End}(\oplus P'_i)^\mr{op}\simeq J_{\sigma'}.
%\end{equation}
\]
\item[(2)]
\[
\Phi_k:=\mathbf{R}\mr{Hom}(\oplus P'_i,\bullet) \colon D^b(\mathrm{mod}J_\sigma)
\to
D^b(\mathrm{mod}J_{\sigma'})
\]
provides an equivalence.
\end{itemize}
\end{thm}
For a $J_\sigma$-module $V=\oplus_{i\in \hI}V_i$, we have
\[
\left(H^j_{\mathrm{mod}J_{\sigma'}}(\Phi_k(V))\right)_i
=
\begin{cases}
V_i & i\neq k, j=0,\\
\mr{ker}\left(V_{k-1}\oplus V_{k+1}\to V_k\right) & i=k, j=0,\\
\mr{coker}\left(V_{k-1}\oplus V_{k+1}\to V_k\right) & i=k, j=1,\\
0 & \text{otherwise.}
\end{cases}
\]
As for dimension vectors, the simple reflection is compatible with the derived equivalence.

%We put
%\begin{align*}
%(\mathrm{mod}J_{\sigma})_k
%&:=
%\{
%V\in \mathrm{mod}J_{\sigma}\mid 
%\mr{Hom}(s_k,V)=0
%\},\\
%(\mathrm{mod}J_{\sigma})^k
%&:=
%\{
%V\in \mathrm{mod}J_{\sigma}\mid 
%\mr{Hom}(V,s_k)=0
%\}.
%\end{align*}
By the description above, we have 
\begin{align*}
\mathrm{mod}J_{\sigma}\cap 
\Phi_k^{-1}(\mathrm{mod}J_{\sigma'})
&=
\{
V\in \mathrm{mod}J_{\sigma}\mid 
\mr{coker}\left(V_{k-1}\oplus V_{k+1}\to V_k\right)=0
\}\\
&=\{
V\in \mathrm{mod}J_{\sigma}\mid 
\mr{Hom}(V,s_k)=0
\}\\
&=: (\mathrm{mod}J_{\sigma})^k,\\
\mathrm{mod}J_{\sigma}\cap 
\Phi_k^{-1}(\mathrm{mod}J_{\sigma'})[1]
&=
\{
V\in \mathrm{mod}J_{\sigma}\mid 
V_i=0\ (i\neq k)
\}\\
&=:\mathcal{S}_k.
\end{align*}
In other-words, $((\mathrm{mod}J_{\sigma})^k,\mathcal{S}_k)$ is a torsion pair of $\mathrm{mod}J_{\sigma}$ and $\Phi_k^{-1}(\mathrm{mod}J_{\sigma'})$ is obtained from $\mathrm{mod}J_{\sigma}$ by tilting with respect to this torsion pair  (see \cite[\S 3.1]{cluster-via-DT}).
Then we have
\begin{align*}
\mathrm{mod}J_{\sigma'}\cap 
\Phi_k(\mathrm{mod}J_{\sigma})
%&
%=\{
%V\in \mathrm{mod}J_{\sigma'}\mid 
%\mr{Hom}(V,\Phi_k(s_k[-1]))=0
%\}\\
&=\{
V\in \mathrm{mod}J_{\sigma'}\mid 
\mr{Hom}(s_k',V)=0
\}\\
&=: (\mathrm{mod}J_{\sigma'})_k.
\end{align*}
In summary, we have the following:
\begin{prop}
The equivalence $\Phi_k$ induces an equivalence of $(\mathrm{mod}J_{\sigma})^k$ and $(\mathrm{mod}J_{\sigma'})_k$.
\end{prop}
%\begin{rk}
%This description provides the same factorization of the generating function of %the motivic DT invariants as the end of \cite[pp23]{motivic_WC}.
%\end{rk}
In the proof of \cite[Proposition 3.1]{3tcy}, the author provides the isomorphism in Proposition \ref{thm_mutation} (1) explicitly.
For $V\in \mathrm{mod}J_{\sigma}\cap 
\Phi_k^{-1}(\mathrm{mod}J_{\sigma'})$, the map 
\[
\left(H^0_{\mathrm{mod}J_{\sigma'}}(\Phi_k(V))\right)_{k-1}
\to 
\left(H^0_{\mathrm{mod}J_{\sigma'}}(\Phi_k(V))\right)_{k}
\]
is induced by the following morphism:
\[
R_{k-1}\oplus R_{k-1,k+1}\colon V_{k-1}\to V_{k-1}\oplus V_{k+1}
\]
where
\[
R_{k-1}:=
\begin{cases}
r_{k-1} & k-1\in \hI_r,\\
h_{k-\frac{3}{2}}^+\circ h_{k-\frac{3}{2}}^- & k-1\notin \hI_r
\end{cases}
\]
and
\[
R_{k-1,k+1}:=
h_{k+\frac{1}{2}}^-
\circ
h_{k+\frac{1}{2}}^+
\circ
h_{k-\frac{1}{2}}^+.
\]

%\begin{cor}
%The equivalence in Theorem \ref{thm_mutation} is given by tilting.
%\end{cor}

\subsection{Cut and mutation}\label{sec_cut}
Let $(Q,W)$ be a quiver with potential. 
To each subset $C\subset Q_1$ we associate a grading $g_C$ on $Q$ by
\[
g_C(a)
=
\begin{cases}
1 & a \in C,\\
0 & a \in C.
\end{cases}
\]
A subset $C\subset Q_1$ is called a cut if $W$ is homogeneous of degree $1$ with respect to $g_C$.
Denote by $Q_C$, the subquiver of $Q$ with the vertex set $Q_0$ and the arrow set $Q_1\backslash C$. We define the truncated Jacobian algebra by
%\begin{align*}
\[
J(Q,W)_C  := J(Q,W)/\langle C\rangle
\]
%\end{align*}

Let $k$ be a vertex of $Q_\sigma$ without loops and $C$ be a cut of $(Q_\sigma,w_\sigma)$ such that $g_C(h_{k+\frac{1}{2}}^+)=1$\footnote{We can construct a cut of $(Q_\sigma,w_\sigma)$ as follows: First, by coupling $h_i^+$ and $h_i^-$ for each $i$, we group the arrows in $Q_\sigma$ into $N+|\hI_r|$ groups. Note that $N+|\hI_r|$ is even. These groups have the natural cyclic order and we label each of them by odd or even. Choose (any) one arrow from each odd (or even) labelled group, then we get a cut.}.
We define a cut $C'$ of $(Q_{\sigma'},w_{\sigma'})$ by the following conditions:
\begin{itemize}
\item $g_{C'}(h_{k-\frac{1}{2}}^+)=1$, and
\item $g_{C'}(h_{i}^\pm)=g_{C}(h_{i}^\pm)$ if $i\neq k-\frac{1}{2},k+\frac{1}{2}$.
\end{itemize}

\begin{prop}(See \cite[Proposition 4.12]{motivic_WC})\label{prop_cut_mutation}
The equivalence $\Phi_k$ induces an equivalence of $(\mathrm{mod}J_{\sigma,C})_k$ and $(\mathrm{mod}J_{\sigma',C'})^k$.
\end{prop}
\begin{proof}
It is enough to show that if $h_{k+\frac{1}{2}}^+$ vanishes on $V$ then $h_{k-\frac{1}{2}}^+$ vanished on $\Phi_k(V)$.

Since $g_C(h^\pm_{k-\frac{1}{2}})=0$, we have 
\begin{itemize}
\item $g_C(r_{k-1})=1$ if $k-1\in \hI_r$, and 
\item $g_C(h^+_{k-\frac{3}{2}})=1$ or $g_C(h^-_{k-\frac{3}{2}})=1$ if $k-1\notin \hI_r$
\end{itemize}
and so $R_{k-1}$ vanishes.
Since $g_C(h_{k+\frac{1}{2}}^+)=1$, we see that $R_{k-1,k+1}$ vanishes.
\end{proof}

%% file: MDT.tex
\section{Motivic Donaldson--Thomas invariants}

\subsection{Motives}\label{subsec_11}
We are working in a version of the ring of motivic weights: let $\MC$ denote the $K$-group of the category of effective Chow motives over $\C$, extended by $\LL^{-\oh}$, where $\LL$ is the Lefschetz motive. 
It has a natural structure of a \la-ring \cite{getzler_mixed,heinloth_note} 
%(see section \ref{lambda} for the definition of a \la-ring) 
with $\si$-operations defined by $\si_n([X])=[X^n/S_n]$ and $\si_n(\cL^\oh)=\cL^{\frac n2}$. 
%There is a dimensional completion \cite{behrend_motivica}
We put
$$\tildeMC=\MC\pser{\cL\inv},$$
which is also a \la-ring. Note that in this latter ring, the elements $(1-\cL^n)$, and therefore the motives of general 
linear groups, are invertible. The rings $\MC\sb\tildeMC$ sit in larger rings $\MC^{\hat{\mu}}\subset\tildeMC^{\hat{\mu}}$ of equivariant motives, 
where $\hat\mu$ is the group of all roots of unity~\cite{looijenga_motivic}.

%The map that sends a smooth projective variety $X$ to its $E$-polynomial
%$$E(X,u,v)=\sum_{p,q\ge0}(-1)^{p+q}\dim H^{p,q}(X,\cC)u^pv^q$$
%can be extended to the ring homomorphism $E:\tildeMC\to\cQ[u,v]\pser{(uv)^{-\oh}}$. This map is a \la-ring homomorphism, where the \la-ring structure on $\cQ[u,v]\pser{(uv)^{-\oh}}$ is given by Adams operations (see section \ref{lambda})
%$$\psi_n(f(u,v))=f(u^n,v^n).$$
%The map $E:\MC\to\cQ[u,v,(uv)^{-\oh}]$ can be further specialized to the Euler number $e:\MC\to\cQ$ by $u\mto1,v\mto1,(uv)^{-\oh}\mto1$.

%\begin{rk}
%Note that the Euler number specialization of $\cL^\oh$ is $\cL^\oh\mto1$. This differs from the conventions of \cite{behrend_motivic}, where the specialization is $\cL^\oh\mto-1$. This difference results from the fact that \cite{behrend_motivic} uses the \la-ring structure on $\MC$ with $\si_n(-\cL^\oh)=(-\cL^\oh)^n$ \cite[Remark 1.7]{behrend_motivic}.
%\end{rk}

Let $f\colon X \to \cC$ be a regular function on a smooth variety $X$.
%and let $X_0 = f^{-1}(0)$ be the central fiber.
Using arc spaces, Denef and Loeser \cite{denef_geometry,looijenga_motivic} define the motivic nearby 
cycle $[\psi_f]\in \lM_\cC^{\hat{\mu}}$ and the motivic vanishing cycle 
\[
[\varphi_f]:=[\psi_f]-[f\inv(0)]\in \lM_{\cC}^{\hat{\mu}}
\]
of~$f$. Note that if $f=0$, then $[\vi_0]=-[X]$. The following result was proved in \cite[Prop.~1.11]{behrend_bryan_szendroi}.

\begin{thm}
\label{thr:bbs}
Let $f:X\to\cC$ be a regular function on a smooth variety X.
Assume that $X$ admits a $\cC^*$-action such that $f$ is $\cC^*$-equivariant \ie $f(tx)=tf(x)$ for $t\in\cC^*$, $x\in X$, and 
such that there exist limits $\lim_{t\to0}tx$ for all $x\in X$. Then
$$[\vi_f]=[f\inv(1)]-[f\inv(0)]\in\lM_\cC\sb\lM_\cC^{\hat\mu}.$$
\end{thm}

Following~\cite{behrend_bryan_szendroi}, we define the {\it virtual motive} of $\crit(f)$ to be
\[
[\crit(f)]_\vir := -(-\cL^\oh)^{-\dim X}[\varphi_f]\in \lM^{\hat{\mu}}_\cC.
\]
For a smooth variety $X$, we put
\[
[X]_\vir:= [\crit(0_X)]_\vir=(-\cL^\oh)^{-\dim X}\cdot [X].
\]

%\begin{rk} The ring $\MC$ is known to be a homomorphic image of the naive motivic 
%ring $K_0({\rm Var}_\C)[\LL^{-\oh}]$. Some of the works cited above work in this ring; the quoted constructions
%and results carry over to $\MC$ under this ring homomorphism. We prefer to work in $\MC$ since that is known to 
%be a \la-ring. 
%\end{rk}

\subsection{Quivers and moduli spaces}\label{subsec_32}
Let $Q$ be a quiver, with vertex set $Q_0$ and edge set $Q_1$.
For an arrow $a\in Q_1$, we denote by $s(a)\in Q_0$ (resp.\ $t(a)\in Q_0$) the vertex at which $a$ starts (resp.\ ends).
We define the Euler--Ringel form $\chi$ on $\cZ^{Q_0}$ by the rule
$$\chi(\al,\be)=\sum_{i\in Q_0}\al_i\be_i-\sum_{a\in Q_1}\al_{s(a)}\be_{t(a)},\qquad \al,\be\in\cZ^{Q_0}.$$
%We define the skew-symmetric bilinear form $\ang{\bullet,\bullet}$ of the quiver $Q$ to be 
%\[
%\ang{\al,\be}:=\chi(\al,\be)-\chi(\be,\al),\qquad \al,\be\in\cZ^{Q_0}.
%\]

Given a $Q$-representation $M$, we define its dimension vector $\udim M\in\cN^{Q_0}$ by $\udim M=(\dim M_i)_{i\in Q_0}$.
Let $\al\in\cN^{Q_0}$ be a dimension vector and let $V_i=\cC^{\al_i}$, $i\in Q_0$.
We define
\[
R(Q,\al):=\bigoplus_{a\in {Q_1}}\Hom(V_{s(a)},V_{t(a)})
\]
and 
\[
\GG_\al:=\prod_{i\in Q_0}\GL(V_i).
\]
Note that $\GG_\al$ naturally acts on $R(Q,\al)$ and the quotient stack
\[
\gM(Q,\al):=[R(Q,\al)/\GG_\al]
\]
gives the moduli stack  of representations of $Q$ with dimension vector $\al$.

Let $W$ be a potential on $Q$, a finite linear combination of cyclic paths in $Q$. Denote by $J=J_{Q,W}$ 
the Jacobian algebra, the quotient of the path algebra $\C Q$ by the two-sided ideal generated by formal partial
derivatives of the potential $W$. Let
\[
f_\al:R(Q,\al)\to\cC
\]
be the $\GG_\al$-invariant function defined by taking the trace of the map associated 
to the potential $W$. As it is now well known~\cite[Proposition 3.8]{segal_a}, a point in the critical locus 
$\crit(f_\al)$ corresponds to a $J$-module. The quotient stack 
\[
\gM(J, \al):=\bigl[\crit(f_\al)/\GG_\al\bigr]
\]
gives the moduli stack of $J$-modules with dimension vector $\al$.

\begin{defn}
\label{central charge}
A {\it central charge} is a group homomorphism $Z:\cZ^{Q_0}\to \C$
such that $$Z(\al)\in\cH_+=\sets{re^{i\pi\vi}}{r>0,0<\vi\le1}$$
for any $\al\in \cN^{Q_0}\ms\set0$. 
Given $\al\in \cN^{Q_0}\ms\set0$, the number $\vi(\al)=\vi\in(0,1]$ such that 
$Z(\al)=re^{i\pi\vi}$, for some $r>0$, is called the phase of \al.
\end{defn}

\begin{defn}
\label{semist}
For any nonzero $Q$-representation (resp.\ $J$-module) $V$, we define $\vi(V)=\vi(\udim V)$.
A $Q$-representation (resp.\ $J$-module) $V$ is said to be $Z$-(semi)stable if for any proper nonzero $Q$-subrepresentation (resp.\ $J$-submodule) $U\sb V$ we have
$$\vi(U)(\le)\vi(V).$$
\end{defn}

\begin{defn}
\label{Z from ze}
Given $\zeta\in\cR^{Q_0}$, define the central charge $Z:\cZ^{Q_0}\to\cC$ by the rule
$$Z(\al)=-\zeta\cdot\al+i\n\al,$$
where $\n\al=\sum_{i\in Q_0}\al_i$.
We say that a $Q$-representation (resp.\ $J$-module) is \ze-(semi)stable if it is $Z$-(semi)stable.
\end{defn}

\begin{rk}
\label{rmr:mu}
Let the central charge $Z$ be as in Definition \ref{Z from ze}.
Define the slope function $\mu:\cN^{Q_0}\ms\set0\to\cR$ by $\mu(\al)=\frac{\ze\cdot\al}{\n\al}$. If $l\sb\cH=\cH_+\cup\set0$ is a ray such that $Z(\al)\in l$ then $l=\cR_{\ge0}(-\mu(\al),1)$. This implies that $\vi(\al)<\vi(\be)$ if and only if $\mu(\al)<\mu(\be)$.
\end{rk}

We say that $\ze\in\cR^{Q_0}$ is \al-generic if for any $0<\be<\al$ we have $\vi(\be)\ne\vi(\al)$.
This condition implies that any \ze-semistable $Q$-representation (resp.\ $J$-module) is automatically \ze-stable.

Let $R_\ze(Q,\al)$ denote the open subset of $R(Q,\al)$ consisting of \ze-semistable representations.
Let $f_{\ze,\al}$ denote the restriction of $f_{\al}$ to $R_\ze(Q,\al)$. 
The quotient stacks
\begin{equation}
\gM_\ze(Q,\al):=\bigl[R_\ze(Q,\al)/\GG_\al\bigl],\qquad
\gM_\ze(J, \al):=\bigl[\crit(f_{\ze,\al})/\GG_\al\bigr]
\label{eq:def moduli stacks}
\end{equation} 
give the moduli stacks of $\ze$-semistable $Q$-representations and $J$-modules with dimension vector~\al.

\subsection{Motivic DT invariants}
Let $(Q,W)$ be a quiver with a potential and let $J=J_{Q,W}$ be its Jacobian algebra. Recall that the degeneracy locus of the function $f_\al:R(Q,\al)\to\cC$
defines the locus of $J$-modules, so that the quotient stack 
\[
\gM(J,\al):=[\crit(f_\al)/\GG_\al]
\]
is the stack of $J$-modules with dimension vector~$\al$. 
We define motivic Donaldson--Thomas invariants by
\[
[\gM(J,\al)]_\vir:=\frac{[\crit(f_\al)]_\vir}{[\GG_\al]_\vir}.
\]
For a stability parameter $\ze$, we define
\begin{equation}\label{eqref_33}
[\gM_\ze(J,\al)]_\vir=\frac{[\crit(f_{\ze_,\al})]_\vir}{[\GG_\al]_\vir}.
\end{equation}
where, as before, $f_{\ze,\al}$ denote the restriction of $f_{\al}:R(Q,\al)\to\cC$ to $R_\ze(Q,\al)$.

\subsection{Generating series of motivic DT invariants}
Let $(Q,W)$ be a quiver with a potential admitting a cut, and let $J=J_{Q,W}$ be its Jacobian algebra.

\begin{defn}
\label{def:series}
We define the generating series of the motivic Donaldson--Thomas invariants of $(Q,W)$ by
$$A_U(y)
=\sum_{\al\in\cN^{Q_0}}[\gM(J,\al)]_\vir \cdot y^\al
=\sum_{\al\in\cN^{Q_0}}
\frac{[\crit(f_\al)]_\vir}{[\GG_\al]_\vir}\cdot y^\al\in\That_Q,$$
the subscript referring to the fact that we think of this series as the universal series.
\end{defn}

Given a cut $C$ of $(Q,W)$, we define a new quiver $Q_C=(Q_0,Q_1\ms C)$. 
Let $J_C$ be the quotient of $\cC Q_C$ by the ideal $$(\dd_C W)=(\dd W/\dd a,a\in C).$$

\begin{prop}\cite[Proposition 1.14]{mmns}\label{prop_reduction}
\label{first reduction}
If $(Q,W)$ admits a cut $C$, then
$$A_U( y )=\sum_{\al\in\cN^{Q_0}}(-\cL^\oh)^{\hi(\al,\al)+2d_I(\al)}\frac{[R(J_C,\al)]}{[\GG_\al]}y^\al,$$
where $d_C(\al)=\sum_{(a:i\to j)\in C}\al_i\al_j$ for any $\al\in\cZ^{Q_0}$.
\end{prop}

The quiver with potential $(Q_\sigma,w_\sigma)$ introduced in $\S$\ref{sec_NCCR} admits a cut (see $\S$\ref{sec_cut}) and Proposition \ref{prop_reduction} can be applied. In the next section we use this to compute the universal series in a specific case.

%% file: s3.tex
\thispagestyle{empty}
\section{The universal DT series : special case}\label{sec_univ1}
Throughout this section we 
%work with a fixed resolution. We 
fix $\sigma$ to be the unique partition defined such that
\[ \hat{I}_r = \{ 0,1,2,3,\ldots N'-1 \}. \]
In other words the partition such that the quiver with potential $(Q_{\sigma},w_{\sigma})$ has loops at the first $N'$ vertices only.
The aim of this section is to prove Theorem \ref{thm_A} for this quiver with potential.

We define three fixed subsets of the vertices
\begin{eqnarray*}
I_1 &:=&  \{0,1,\ldots , N'-1\} \subset \mathbb{Z}/N, \\
I_2 &:=&  \{N',N'+2,N'+4,\ldots, N-2\} \subset \mathbb{Z}/N, \\
I_3 &:=&   \{N'+1,N'+3,N'+5,\ldots, N-1\} \subset \mathbb{Z}/N.
\end{eqnarray*}
Then there exists a cut $C$ given by the collection of arrows
\[ C = \left\{ h_i^- \mid i-\half \not \in I_2  \right\}. \]
By Proposition \ref{prop_reduction} the coefficients of the universal DT series $A^{\sigma}_U(y) = \sum_{\alpha\in \mathbb{N}^{Q}} A_\alpha y^\alpha$ are given by
\[ A_\alpha = \left( -\mathbb{L}^\half \right)^{\chi(\alpha,\alpha) +2d_C(\alpha)} \frac{[R(J_{\sigma,C}, \alpha )]}{[G_\alpha]}y^\alpha \]
where $d_C(\alpha) = \sum_{(a:i\to j) \in C} \alpha_i\alpha_j$. To begin we find a simple expression for the term $\chi(\alpha,\alpha)+ 2d_C(\alpha)$ in the exponent.  We know by definition that
\begin{eqnarray*}
\chi(\alpha, \alpha) &=& \sum_{i\in I_1\cup I_2 \cup I_3} \alpha_i^2 - \sum_{i\in I_1} \alpha_i^2- \sum_{i\in I_1\cup I_2 \cup I_3} \alpha_i \alpha_{i+1}  - \sum_{i\in I_1\cup I_2 \cup I_3} \alpha_{i+1} \alpha_{i},\\
 d_I(\alpha) &=& \sum_{i\in I_1} \alpha_i \alpha_{i+1} + 
\sum_{i\in I_3} \alpha_{i+1}\alpha_{i}, 
\end{eqnarray*}
so it follows
\begin{eqnarray*}\chi(\alpha, \alpha) + 2d_C(\alpha) &=& \sum_{i\in I_2\cup I_3} \alpha_i^2 - 2\cdot \sum_{i\in I_2} \alpha_i\alpha_{i+1}, \\
& = &  \sum_{i\in I_2} (\alpha_{i+1}-\alpha_{i})^2.
\end{eqnarray*}
Our next goal is to factorize $A_{U}^{\sigma}(y)$ into two simpler series. This proceeds by analyzing the motivic classes $[R(J_{\sigma,C}, \alpha )]$. 
\\ \\
Given a dimension vector $\alpha \in \mathbb{N}^{Q_0}$ and a representation of a $J_{\sigma,C}$-module
\[ V = \bigoplus_{i\in I_1 \cup I_2 \cup I_3} V_i \]
we focus on a specific element
\[ H:=  h^+_{1/2} + h^+_{3/2} + \cdots + h^+_{N-1/2} \in  \bigoplus_{i\in I_1 \cup I_2 \cup I_3} \Hom (V_i,V_{i+1}).\]
This map $H$ acts as an endomorphism of the vector space $V$. Given any such linear map
\[ H: V\to V \]
there exists a unique splitting $V = V^I \oplus V^N$ with maps
\[\begin{array}{llll}
H^I &:& V^I \to V^I  &\textrm{ invertible}\\
H^N &:& V^N \to V^N &\textrm{ nilpotent}
\end{array}\]
so that 
\[H= H^I\oplus H^N.\]
Moreover in our case the above splitting respects the grading by $i\in I_1\cup I_2\cup I_3$. To be explicit we have that 
\[V^I = \bigoplus_{i \in I_1 \cup I_2 \cup I_3} V_i^I \] 
where $V_i^I := V_i \cap V^I$ (similarly $V^N = \bigoplus_{i \in I_1 \cup I_2 \cup I_3} V_i^N $ with $V_i^N := V_i \cap V^N$). One immediate consequence of this is that
\[ \dim \left( V^I_i \right) = \dim \left( V^I_{i+1} \right) \textrm{ for all } i\in I_1\cup I_2 \cup I_3, \]
indeed this is clear since the block form of $H^I$ demands that it map $V_i^I$ to $V_{i+1}^I$ via an isomorphism. We are now ready to decompose the computation of $A^{\sigma}_{U}(y)$ into two simpler subproblems.
\begin{defn}\label{invdefn}(Invertible series) We define
\[R^I(a) := \{ r \in R(J_{\sigma,C}, \alpha ) \mid H \textrm{ is invertible, } \alpha_i = a \textrm{ }\forall i \}\]
and the series
\[ I^{\sigma}(y) :=  \sum_{a\geq 0}  \frac{ [R^I(a)] }{ [\GL(a)]^N} y^a. \]
\end{defn}
\begin{defn}\label{nildefn}(Nilpotent series) We define
\[R^N(\alpha) := \{ r \in R(J_{\sigma,C}, \alpha ) \mid H \textrm{ is nilpotent}\} \]
and the series
\[ N^\sigma(y) :=  \sum_{\alpha \in \mathbb{N}^{Q_0}  } (-\mathbb{L}^{1/2})^{\sum_{i\in I_2} \left( \alpha_{i+1}-\alpha_i \right)^2 } \frac{[R^N(\alpha)]}{[G_\alpha]} y^\alpha. \]
\end{defn}
The following lemma shows that the series $A_{U}^\sigma(y)$ factorizes into the product of the two just defined.
\begin{lem}\label{factor} We have
\[ A^\sigma_U(y) = I^\sigma(y_0\cdots y_{N-1}) \cdot N^\sigma(y).\]
\end{lem}
\begin{proof} 
This formula follows directly from a stratification of the variety $R(J_{\sigma,C}, \alpha )$ by the dimension of $V_i^I$. 

Fix $\alpha \in \mathbb{N}^{Q_0}$, we stratify $R(J_{\sigma,C}, \alpha )$ by $\dim ( V_i^I ) = a$. Let 
\[ \underline{a} := (a,a,\ldots,a) \in \mathbb{N}^{Q_0}, \]
and
\[ \alpha' \textrm{ such that } \alpha = \underline{a} + \alpha' \in \mathbb{N}^{Q_0}.  \]
There is a Zariski locally trivial fibration

\[\begin{array}{ccc} 
R^I(a) \times R^N(\alpha') & \to & \{ r\in R(J_{\sigma,C}, \alpha ) \mid\dim ( V_i^I ) = a \textrm{ for } H\in r \}\\
& & \downarrow  \\
 & & \mathcal{M}(a,\alpha).
\end{array}\]

Here $\mathcal{M}(a,\alpha)$ is the space parameterizing splittings $V_i = V^I_i \oplus V_i^N$. To see this one checks that the arrows $r_i$, $h^-_{i+1/2}$ in the representation also preserve the splitting, so the entire representation splits into $V^I\oplus V^N$. This follows easily from the relations and some basic linear algebra. 

Splittings of the vector space $V_i = V_i^I \oplus V_i^N$ are parameterized by 
\[ \GL(\alpha_i) / \left( \GL(a) \times \GL(\alpha'_i) \right) \]
and hence the motivic class of the base is
\[ [\mathcal{M}(a,\alpha)] = \frac{[G_\alpha]}{[\GL(a)]^{N}\cdot [G_{\alpha'}] }. \]
Summing over each stratum with $\dim (V_i^I) = a$ we get
\[ [R(J_{\sigma,C}, \alpha )] = [G_\alpha] \cdot \sum_{a=0}^{\min_i \{\alpha_i\}} \frac{[R^I(a)]}{[\GL(a)]^N} \cdot \frac{[R^N(\alpha ')]}{[G_{\alpha'}]}.  \]
Multiplying both sides of this expression by $ (-\mathbb{L}^{1/2})^{\sum_{i\in I_2} \left( \alpha_{i+1}-\alpha_i \right)^2 } y^{\alpha}$ and summing gives
\[
A_U^\sigma(y)
= \left(  \sum_{a\geq 0}  \frac{ [R^I(a)] }{ [\GL(a)]^N} \prod_{i=0}^{N-1}y_i^a \right) 
\cdot \left( \sum_{\alpha' \in \mathbb{N}^{Q_0}  }(-\mathbb{L}^{1/2})^{\sum_{i\in I_2} \left( \alpha'_{i+1}-\alpha'_i \right)^2 } \frac{[R^N(\alpha')]}{[G_{\alpha'}]} y^{\alpha'} \right)
\]
proving the result.
\end{proof}
In the next two sections we compute formulas for $I^\sigma(y)$ and $N^\sigma(y)$.
\subsection{Step One: The invertible case $I^\sigma(y)$}
\begin{prop}\label{invseries} We have
\[I^\sigma(y) = \Exp \left( {\mathbb{L}} \frac{y}{1-y}  \right) . \]
\end{prop}
\begin{proof} A $J_{\sigma,C}$-module $r \in R(J_{\sigma,C},\alpha)$ is given by a vector space 
\[V =\bigoplus_{i\in I_1\cup I_2 \cup I_3} V_i \] 
of dimension $\alpha \in \mathbb{N}^{Q_0}$ and a collection of linear maps
\[\begin{array}{llll}
r_i &:& V_i \to V_i                     &\textrm{ for } i \in I_1 \\
h^-_{i+1/2} &:&V_{i+1}\to V_i &\textrm{ for } i \in I_2 \\
h^+_{i+1/2}&:&V_i \to V_{i+1}&\textrm{ for } i \in I_1 \cup I_2 \cup I_3
\end{array}\]
satisfying the relations coming from cyclic differentiation of the potential
\[\begin{array}{llll}
r_i h^+_{i-1/2} &=& h_{i-1/2}^+r_{i-1} &\textrm{ for } i\in [1,N'-1]\cap I_1 \\
r_0h^+_{N-1/2} &=&h^+_{N-1/2} h^+_{N-3/2}  h^-_{N-3/2} &\\
h^-_{N'+1/2} h^+_{N'+1/2}h^+_{N'-1/2}&=& h^+_{N'-1/2}r_{N'-1} & \\
h^-_{i+3/2} h^+_{i+3/2}h^+_{i+1/2}&=&h^+_{i+1/2} h^+_{i-1/2}  h^-_{i-1/2}& \textrm{ for } i=[N'+1,N-3 ]\cap I_3 .
\end{array}\]
assuming moreover that $r\in R^I(a)$ then
\[ h^+_{i+1/2} : V_i\to V_{i+1} \textrm{ is invertible } \forall i\in I_1\cup I_2\cup I_3. \]
This allows us to express $R^{I}(a)$ as a $\prod_{i=1}^{N-1}\GL (V_i)$ torsor over a commuting variety
\[\begin{array}{lllll}
\pi &:& R^I(a) &\to& C(a) \\
&:& (r_i,h^+_{i+1/2},h^-_{i+1/2}) & \mapsto & (r_0,h^+_{N-1/2}h^+_{N-3/2}\cdots h^+_{3/2}h^+_{1/2})
\end{array}\]
where
\[ C(a) = \{ (A,B) \in \End(V_0) \times \GL(V_{0}) \mid AB=BA \}. \] 
The free action of $\prod_{i=1}^{N-1}\GL (V_i)$ on $R^I(a)$ is given by
\[\begin{array}{llll}
(g_1,\ldots, g_{N-1})&:& r_i \mapsto  g_i r_i g_i^{-1} & \textrm{ for } i \in [1,N'-1]  \\
 &:& h^{+}_{1/2} \mapsto  g_1h^+_{1/2} & \\
 &:& h^{+}_{N-1/2} \mapsto  h^+_{N-1/2} g_{N-1}^{-1} & \\
 &:& h^{+}_{i+1/2} \mapsto  g_{i+1}h^+_{i+1/2} g_{i}^{-1} & \textrm{ for } i \in [1,N-2] \\
 &:& h^{-}_{i+1/2} \mapsto  g_{i}h^-_{i+1/2} g_{i+1}^{-1} & \textrm{ for } i \in I_2.
\end{array}\] 
As $\GL(a)$ is a special group the torsor splits in the Zariski topology, motivically we have 
\[ [R^I(a)] = [\GL(a)]^{N-1}\cdot [C(a)]. \]
Thus 
\[ I^\sigma(y) = \sum_{a\geq 0} \frac{[C(a)]}{[\GL(a)]} y^a. \]
The generating series for the commuting variety is obtained in \cite{bm} giving the result.
\end{proof}
%Here starts step two.

\subsection{Step Two: The nilpotent case $N^\sigma(y)$}\label{nil}

This section is the final step in the calculation. Here we compute $N^\sigma(y)$ and obtain the formula of $A_U^\sigma(y)$.

We fix a dimension vector $\alpha \in \mathbb{N}^{Q_0}$. As before a $J_{\sigma,C}$-module is given by a vector space 
\[V =\bigoplus_{i\in I_1\cup I_2 \cup I_3} V_i \] 
of dimension $\alpha$ and a collection of linear maps
\[\begin{array}{llll}
r_i &:& V_i \to V_i                     &\textrm{ for } i \in I_1 \\
h^-_{i+1/2} &:&V_{i+1}\to V_i &\textrm{ for } i \in I_2 \\
h^+_{i+1/2}&:&V_i \to V_{i+1}&\textrm{ for } i \in I_1 \cup I_2 \cup I_3
\end{array}\]
satisfying the relations of the potential (see Proposition \ref{invseries}). Throughout this section we insist that the map
\[ H = h^{+}_{1/2} + h^{+}_{3/2} + \cdots h^{+}_{N-1/2} \in \bigoplus_{i\in I_1\cup I_2\cup I_3} \Hom (V_i,V_{i+1}) \]
is nilpotent. In fact $R^N(\alpha)$ is exactly the collection of all such representations (see Definition \ref{nildefn}). In particular if we let $|\alpha| : = \dim(V)$ then we know that $H^{|\alpha|} = 0$. This gives a filtration of the vector space,
\[ V = V^{|\alpha |} \supset V^{|\alpha |-1} \supset \cdots \supset  V^{1} \supset V^{ 0 } = \{ 0 \} \]
where
\[ V^j = \{ v \in V \mid H^j(v) =0  \}. \]
Moreover the filtration respects the grading by $i\in I_1 \cup I_2 \cup I_3$, by which we mean that
\[ V^j = \bigoplus_{i\in I_1 \cup I_2 \cup I_3 } \left(V^j \cap V_i\right)   \]
where $V_i$ is the summand at the $i$th vertex of the quiver. By considering the vector space $V$ as a representation of the nilpotent matrix $H$ we can identify $V$ with a $\mathbb{C}[x]$-module supported at the origin. Modules for a principal ideal domain have a simple structure. In particular we have
\[ V \cong \bigoplus_{j=1}^d \left( \mathbb{C}[x]/(x^j) \right)^{\oplus b_j} \]
as a $\mathbb{C}[x]$-module. The next proposition provides a more refined version of this statement where each factor in this decomposition is generated by a vector from a vector space $V_i$. 
\begin{prop}\label{decomp}For each $i\in I_1 \cup I_2 \cup I_3$ there exists  collection of integers $b_j^i$ so that
\[ V\cong \bigoplus_{i\in I_1 \cup I_2 \cup I_3} \bigoplus_{i=1}^{d} \left( \mathbb{C}[x]/(x^j) \right)^{ \oplus b_j^i}  \]
where the factor $\left( \mathbb{C}[x]/(x^j) \right)^{ \oplus b_j^i}$ is generated as a $\mathbb{C}[x]$-module by vectors in $V_i$. Moreover the numbers $b_j^i$ are uniquely determined by the above conditions.
\end{prop}
\begin{proof}We will argue by induction on $d$, the largest integer such that $b_d\neq 0$. As such we can assume that for each $j\leq d-1$ the factor $ \mathbb{C}[x]/(x^j) $ is generated by a vector in some $V_i$. Now let $e_1,\ldots ,e_{b_d}$ be a generating set for the factor $\left(\mathbb{C}[x]/(x^d) \right)^{\oplus b_d}$, and define $W : = \textrm{span}\{ e_1,\ldots ,e_{b_d} \}$. We consider the projection operators
\[ p_i : V \to V_i / V_i\cap V^{d-1}\]
and set $W_{i} := p_i (W)$ and $b^i_d = \dim W_{i}$. We claim that $p_0\oplus\cdots \oplus p_{N-1}:W \to W_0\oplus \cdots \oplus W_{N-1}$ is an isomorphism. The map is clearly onto and an injection since any vector in the kernel must lie in $V^{d-1}$. Now consider a lifting of the vector space $V_i \supset W'_i \surj W_i\subset V_i / V_i\cap V^{d-1}$ then
\[ W'_i \oplus H W'_i \oplus \cdots \oplus H^{d-1} W'_i \subset V \]
is a submodule of $V$ isomorphic to $\left( \mathbb{C}[x] / (x^d)\right)^{\oplus b^i_d}$. Summing over all $i$ we have that $\left( \mathbb{C}[x]/(x^d) \right)^{\sum_i b^i_d}$ is a submodule of $V$, hence it follows that $\sum_i \dim W_i = \sum_i  b^i_d \leq b_d = \dim W$ and so for dimension reasons we get 
\[ V\cong \left( \bigoplus_{i=0}^{N-1} \left( \mathbb{C}[x] / (x^d) \right)^{\oplus b^i_d}\right) \bigoplus \left( \bigoplus_{j=1}^{d-1} \left( \mathbb{C}[x]/(x^j) \right)^{\oplus b_j} \right). \] 
Here each factor $\left( \mathbb{C}[x] / (x^d) \right)^{\oplus b^i_d}$ is generated by vectors in $V_i$, so by our inductive hypothesis the entire module is generated by vectors in $V_i$.

Finally we prove the uniqueness statement. Assume we have two distinct such decompositions
\[ V\cong \bigoplus_{i=0}^{N-1}\bigoplus_{j=1}^{d}( \mathbb{C}[x]/(x^j) )^{\oplus b^i_j} 
\cong \bigoplus_{i=0}^{N-1}\bigoplus_{j=1}^{d}( \mathbb{C}[x]/(x^j) )^{\oplus c^i_j} . \] By restricting to subrepresentations if necessary we can assume that $b^i_d \neq c^i_d$ for some $i$. However in this case
\[ b_d^i = \dim\left( \ker(H^d : V_i \to V_{i+d}) / V_i \cap V^{d-1}    \right) = c_d^i \]
is a contradiction. This proves the last part of the lemma.
\end{proof}
Next we organize this data in the way most helpful to our cause.
\begin{defn} Let $0\leq a,b \leq N-1$. We define
\[ |b-a|  = \min \{  r \in \{ 0,1,\ldots,N-1 \} \mid b=a+r \mod N \} . \]
\end{defn}
Intuitively this is the distance from $a$ to $b$ in the cyclic direction $i\to i+1$ corresponding to the map $H$.
\begin{defn}Suppose we have a decomposition of $V$ as a $\mathbb{C}[x]$-module as in Proposition \ref{decomp} . Define $V^{a,b}$ to be the vector subspace corresponding to summand 
\[\bigoplus_{l\geq 1} \left(\mathbb{C}[x]/(x^{N(l-1) + |b-a| +1}) \right)^{b^a_{N(l-1) + |b-a| +1}}.\]
and relabel the integers
\[ b^{a,b}_l :=  b^a_{N(l-1) + |b-a| +1} ,\]
to define partitions
\[ \pi^{[a,b]} := (1^{b_1^{a,b}}2^{b^{a,b}_2} 3^{b^{a,b}_3}\cdots).\] 
\end{defn}
Notice that the above definition depends on the choice of the decomposition in Proposition \ref{decomp}. However all such vector spaces are isomorphism abstractly as $\mathbb{C}[x]$-modules. We can think of these vector spaces as being generated by the nilpotent vectors that start at the $a$th vertex and are annihilated at the $b+1$th vertex under the action of the map $H$. 

The next lemma makes explicit how to recover the dimension vector of a representation from the datum of the $N^2$ partitions $\{\pi^{[a,b]}\mid 0\leq a,b \leq N-1 \}$.
\begin{lem}\label{dimv}Given a representation $r\in R^N(\alpha)$ so that the endomorphism $H$ has type $\{ \pi^{[a,b]}\}$ the dimension vector of the representation $r$ is given by 
\[ \alpha_i = \sum_{a,b} |\pi^{[a,b]}| - \sum_{a,b : i\not\in [a,b]} l(\pi^{[a,b]}) \]
where $|\pi^{[a,b]}|$ and $l(\pi^{[a,b]})$ are the size and length of the partition $\pi^{[a,b]}$. \end{lem}
\begin{proof}This is clear since
\[ V = \bigoplus_{a,b} V^{a,b}  \]
and
\[ \dim \left( V^{a,b} \cap V_i \right) = \left\{ \begin{array}{ll} |\pi^{[a,b]}| & \textrm{ if } i \in [a,b] \\ |\pi^{[a,b]}| - l(\pi^{[a,b]}) & \textrm{ if } i \not\in [a,b] . \end{array} \right.  \]
\end{proof}
We can use this to give a simple reformulation of the term $\chi(\alpha,\alpha) + 2d_C(\alpha)$ appearing in the series $N^\sigma$.
\begin{cor}We have
\[\chi(\alpha,\alpha) + 2d_C(\alpha) = \sum_{i\in I_2} \left( \sum_{ b \neq i } l(\pi^{[i+1,b]}) - \sum_{ c\neq i+1 }l(\pi^{[c,i]})\right)^2 . \]
\end{cor}
\begin{proof}In our initial analysis of these terms we saw that
\[ \chi(\alpha,\alpha) + 2d_C(\alpha) = \sum_{i\in I_2} \left( \alpha_{i+1} -\alpha_i \right)^2  \]
and now by lemma \ref{dimv} we have
\[ \alpha_{i+1} - \alpha_i =  \sum_{ b \neq i } l(\pi^{[i+1,b]}) - \sum_{ c\neq i+1 }l(\pi^{[c,i]}). \]
\end{proof} 
The above classification has been for the purpose of breaking the variety $R^N(\alpha)$ down into simpler pieces.
\begin{defn}Given $N^2$ partitions $\{\pi^{[a,b]} \mid 0 \leq a,b \leq N-1\}$ and a dimension vector $\alpha$ as in lemma \ref{dimv} we define
\[ R(\{\pi^{[a,b]}\}) = \{ r\in R^{N}(\alpha) \mid H \textrm{ has type } \{\pi^{[a,b]}\} \}. \]
\end{defn}
This provides a stratification of  $R^{N}(\alpha)$ into strata where the normal form of $H$ has a fixed type. We will proceed to compute the motivic classes of each of these stratum. 

A representation in $R(\{\pi^{[a,b]}\})$ is given explicitly by a vector space $V = \bigoplus_{i\in I_1 \cup I_2 \cup I_3} V_i$ and a collection of linear maps corresponding to the arrows $r_i$ with $i\in I_1$ ,$h^-_{i+1/2}$ with $i\in I_2$ and $h^+_{i+1/2}$ with $i\in I_1\cup I_2 \cup I_3$. In addition the linear maps satisfy relations
\[\begin{array}{llll}
r_i h^+_{i-1/2} &=& h_{i-1/2}^+r_{i-1} &\textrm{ for } i\in [1,N'-1]\cap I_1 \\
r_0h^+_{N-1/2} &=&h^+_{N-1/2} h^+_{N-3/2}  h^-_{N-3/2} &\\
h^-_{N'+1/2} h^+_{N'+1/2}h^+_{N'-1/2}&=& h^+_{N'-1/2}r_{N'-1} & \\
h^-_{i+3/2} h^+_{i+3/2}h^+_{i+1/2}&=&h^+_{i+1/2} h^+_{i-1/2}  h^-_{i-1/2}& \textrm{ for } i=[N'+1,N-3 ]\cap I_3 .
\end{array}\]
and we require that the map
\[ H = h^{+}_{1/2} + h^{+}_{3/2} + \cdots h^{+}_{N-1/2} \in \bigoplus_{i\in I_1\cup I_2\cup I_3} \Hom (V_i,V_{i+1}) \]
has a type given by the partitions $\{ \pi^{[a,b]} \mid 0\leq a,b \leq N-1 \}$. The linear map $H$ contains all the information of the maps $h^+_{i+1/2}$. For brevity we make the following definition packaging all the remaining linear maps into one.
\begin{defn}Given a representation as above, we define the linear map
\[ L : = r_0 + r_1 + \cdots r_{N'-1} + h^-_{N'+1/2} + \cdots h^-_{N-3/2} \in \bigoplus_{i\in I_1} \Hom(V_i,V_i) \bigoplus_{i\in I_2} \Hom(V_{i+1},V_i). \]
\end{defn}
From now on in order to compute the motivic class of $R(\{ \pi^{[a,b]} \})$ we will work with a choice of coordinates. Let 
\[ v^{a,b}_l(k) \in V_a \]
be such that $v^{a,b}_l(k)$ generates the $k$th summand of $\mathbb{C}[x]/(x^{N(l-1)+ |b-a| +1})^{\oplus b^{a,b}_l}$ in the decomposition of Proposition \ref{decomp}. Then we have that
\[ \mathcal{B } : =  \{ H^{p}v^{a,b}_l(k) \mid  1\leq k \leq b^{a,b}_l,  0\leq a,b \leq N-1, 0 \leq p \leq N(l-1)+ |b-a| +1 \} \]
forms a basis of $V$.
\begin{defn}We define $H( \pi^{[a,b]})$ to be the matrix representation of the map $H$ with respect to the basis $\mathcal{B}$. Also define
\begin{eqnarray*}
F(\{\pi^{[a,b]}\}) &:=& \{ L \mid (L,H(\pi^{[a,b]})) \in R(\{ \pi^{[a,b]} \}) \} \\
N(\{\pi^{[a,b]}\}) &:=& \{ H \mid H \textrm{ has type } \{ \pi^{[a,b]} \}\}.
\end{eqnarray*}
\end{defn}
Then $ R(\{ \pi^{[a,b]} \})$ has a decomposition as a vector bundle.
\begin{lem}\label{Risvbun}$R(\{ \pi^{[a,b]} \})$ has the structure of a vector bundle
\[ \begin{array}{ccc} 
F(\{ \pi^{[a,b]} \}) & \to & R(\{ \pi^{[a,b]} \}) \\
 & & \downarrow \\
 & & N(\{ \pi^{[a,b]} \}).
\end{array} \]
In particular we have that 
\[ [R(\{ \pi^{[a,b]} \})] = [F(\{ \pi^{[a,b]} \})] \cdot [N(\{ \pi^{[a,b]} \})] \]
in the Grothendieck ring of varieties.
\end{lem}
\begin{proof} The projection map
\[
\begin{array}{ccccc} 
  p    & : & R(\{ \pi^{[a,b]} \}) & \to & N(\{ \pi^{[a,b]} \}) \\
              & : & (L,H) & \mapsto & H .
 \end{array}
 \]
 defines the bundle structure with zero section $H \mapsto (0,H)$. The fibre is the linear space of all such $L$.
\end{proof}
Here the base of the vector bundle is the space of all matrices of type $\{ \pi^{[a,b]} \}$ these are all conjugate to $H(\pi^{[a,b]})$, therefore we have a torsor,
\[
\begin{array}{ccccc} 
\pi &:& G_{\alpha}& \to & N(\pi^{[a,b]})\\
     &:& P                 & \mapsto & P H(\pi^{[a,b]}) P^{-1}.
      \end{array} 
 \]
 This is a torsor for the group $S'(\{\pi^{[a,b]}\}) := \textrm{Stab}_{G_\alpha}(H(\pi^{a,b}))$. This group is given as the group of units in an algebra.
 \begin{defn} We identify $S'(\{\pi^{[a,b]}\})$ with the group of multiplicative units in the following algebra
 \[ S(\{\pi^{[a,b]}\}) := \left\{ N \in \prod_{i=0}^{N-1}\End(\alpha_i) \Big|\ NH(\pi^{[a,b]}) = H(\pi^{[a,b]})N \right\}. \]
 \end{defn}
 Since $S'(\{\pi^{[a,b]}\})$ is the group of units of an algebra it is a special group and so the above torsor splits in the Zariski topology.  The next lemma gives a formula of the motivic class of the group $S'(\{\pi^{[a,b,]}\})$ and via the splitting of the above torsor we deduce a formula for the class of $N(\{ \pi^{[a,b]} \})$. Before stating the lemma we create some notation.
 \begin{defn}The following linear spaces have dimension
 \begin{eqnarray*} 
 T(\{\pi^{[a,b]}\}) &:=& \dim F(\{ \pi^{[a,b]} \}) \\
 B(\{\pi^{[a,b]}\}) &:=& \dim S(\{ \pi^{[a,b]} \}) .
 \end{eqnarray*} 
 \end{defn}
 \begin{lem}\label{ReqTmB}We have
 \[ [S'(\{\pi^{[a,b]}\})] = [S(\{\pi^{[a,b]}\})] \cdot \prod_{0\leq a,b \leq N-1}  \frac{1}{f\left(\pi^{[a,b]}\right)}\]
 where 
 \[f\left(\pi^{[a,b]}\right) := \prod_{l\geq 1} \frac{[\End(b_l^{a,b})]}{[\GL(b^{a,b}_l)]}. \]
So as a consequence
\[ [R(\{ \pi^{a,b} \})] =  [G_{\alpha}] \cdot  \mathbb{L}^{ T(\{\pi^{[a,b]}\}) - B(\{\pi^{[a,b]}\}) } \cdot \prod_{0\leq a,b \leq N-1} f\left(\pi^{[a,b]}\right) .\]
 \end{lem}
 \begin{proof} Let 
 \[ W^{a,b}_l := \textrm{span}_{\mathbb{C}}\{v^{a,b}_l(k)\mid 1\leq k\leq b^{a,b}_l\} \]
 be the span of the basis elements $v^{a,b}_l(k)$ for $1\leq k\leq b^{a,b}_l$. We have both inclusion and projection
 \[ W^{a,b}_l \inj V \surj W^{a,b}_l. \] 
 This gives a map of algebras
 \begin{eqnarray*} 
 \pi &:& S(\{\pi^{[a,b]}\}) \to \prod_{a,b,l} \End(W^{a,b}_l) \\
   &:& N \mapsto \oplus_{a,b,l} N|_{W^{a,b}_l}.
 \end{eqnarray*} 
 This splits as a trivial vector bundle, whose rank is the dimension of the total space minus the dimension of the base. Since we have that the group $S'(\{\pi^{[a,b]}\})$ is the group of units in $S(\{\pi^{[a,b]}\})$, we can identify $S'(\{\pi^{[a,b]}\})$ as the inverse image of the units on the right hand side. This is a trivial vector bundle of rank equal to $ \dim S(\{\pi^{[a,b]}\}) - \dim \prod_{a,b,l}\End (W^{a,b}_l)$. We have an isomorphism of varieties
 
 \[ S'(\{\pi^{[a,b]}\}) \equiv  \frac{ S(\{\pi^{[a,b]}\}) }{ \prod_{a,b,l} \End (W^{a,b}_l)} \times  \prod_{a,b,l} \GL(W^{a,b}_l) \]
 so motivically we have
 \[ [S'(\{\pi^{[a,b]}\})] = [S(\{\pi^{[a,b]}\})] \cdot \prod_{0\leq a,b \leq N-1}  \frac{1}{f\left(\pi^{[a,b]}\right)}.\]
 In lemma \ref{Risvbun} we saw that 
 \[ [R(\{\pi^{[a,b]}\})] = [F(\{\pi^{[a,b]}\})] \cdot [N(\{\pi^{[a,b]}\})] \]
 Now we know that $N(\{\pi^{[a,b]}\})$ is a torsor for the group $S'(\{\pi^{[a,b]}\})$ whose motive we have just computed we deduce
 \begin{eqnarray*}
 [R(\{\pi^{[a,b]}\})] &=& [F(\{\pi^{[a,b]}\})] \cdot \frac{[G_\alpha]}{[S'(\{\pi^{[a,b]}\})]} \\
    &=& [F(\{\pi^{[a,b]}\})] \cdot \frac{[G_\alpha]}{[S(\{\pi^{[a,b]}\})]}  \cdot \prod_{0\leq a,b \leq N-1} f\left( \pi^{[a,b]} \right) \\
     &=& [G_\alpha] \cdot \mathbb{L}^{T(\{\pi^{[a,b]}\}) -B(\{\pi^{[a,b]}\})}  \cdot \prod_{0\leq a,b \leq N-1} f\left( \pi^{[a,b]} \right).
 \end{eqnarray*}
 \end{proof}

 The next proposition computes the difference $ T(\{\pi^{[a,b]}\}) - B(\{\pi^{[a,b]}\})$. Its proof is found in the appendix.
 \begin{prop}\label{difference}We have $ T(\{\pi^{[a,b]}\}) - B(\{\pi^{[a,b]}\}) $ equals to
\[-\frac{1}{2}\sum_{i\in I_2} \left( \sum_{ b \neq i } l(\pi^{[i+1,b]}) - \sum_{ c\neq i+1 }l(\pi^{[c,i]}) \right)^2 
- \half \sum_{a\in I_3 , b\not\in I_2} \sum_{i\geq 1}(b_i^{a,b})^2 - \half\sum_{a\not\in I_3 , b\in I_2}  \sum_{i\geq 1} (b_i^{a,b})^2. 
\]
 \end{prop}
 \begin{proof}The proof is a linear algebra calculation. See appendix. 
 \end{proof}
 As a corollary we deduce the formula for $N^\sigma(y)$.
 \begin{prop} Let 
 \[\begin{array}{lll} S &=& \{ [a,b] \mid a\in I_3 , b\not \in I_2 \textrm{ or } a\not \in I_3, b\in I_2 \} ,\\
y_{[a,b]} &=& y_a\cdot y_{a+1} \cdots y_{b}, \\
y' &=&  y_0\cdot y_1\cdots y_{N-1}, \end{array}
 \]
then we have 
\[N^\sigma(y) = \Exp \left( \frac{\mathbb{L}}{\mathbb{L}-1}\frac{1}{1-y'}\left(
\sum_{[a,b] \not\in S} y _{[a,b]} -\mathbb{L}^{-\half}\sum_{[a,b] \in S} y_{[a,b]}\right) \right).\]
 \end{prop}
 \begin{proof}Recall our initial definition of $N^\sigma(y)$
\[ N^\sigma(y) = \sum_{\alpha \in \mathbb{N}^{Q_0}  } (-\mathbb{L}^{1/2})^{\chi(\alpha, \alpha) + 2d_C(\alpha) } \frac{[R^N(\alpha)]}{[G_\alpha]} y^\alpha \]
In Proposition \ref{decomp} we saw that it was possible to stratify each of the varieties $R^N(\alpha)$ by the type $\{\pi^{[a,b]} \}$ of the cycle $H$. This gives
\[ N^\sigma(y) = \sum_{\alpha \in \mathbb{N}^{Q_0}  } (-\mathbb{L}^{1/2})^{\chi(\alpha, \alpha) + 2d_C(\alpha) }[G_\alpha]^{-1}\left( \sum_{\{ \pi^{[a,b]} \} \vdash \alpha} [R(\{\pi^{[a,b]}\})] \right)  y^\alpha .\]
The motivic class of $R(\{\pi^{[a,b]}\})$ was computed in lemma \ref{ReqTmB} substituting this into the above formula gives
\[  \sum_{\alpha \in \mathbb{N}^{Q_0}  } (-\mathbb{L}^{1/2})^{\chi(\alpha, \alpha) + 2d_C(\alpha) }\left( \sum_{\{ \pi^{[a,b]} \} \vdash \alpha} \mathbb{L}^{T(\{ \pi^{[a,b]} \}) - B(\{ \pi^{[a,b]} \})} \cdot \prod_{0\leq a,b \leq N-1} f\left( \pi^{[a,b]} \right) \right)  y^\alpha . \]
Lemma \ref{dimv} showed how the dimension vector depended on the partitions we had
\[ \alpha_i = \sum_{0\leq a,b \leq N-1}|\pi^{[a,b]}| - \sum_{[a,b]\not\ni i} l(\pi^{[a,b]}) \]
and an immediate corollary was that
\[ \chi(\alpha, \alpha) + 2d_C(\alpha) = \sum_{i\in I_2} \left( \sum_{ b \neq i } l(\pi^{[i+1,b]}) - \sum_{ c\neq i+1 }l(\pi^{[c,i]}) \right)^2. \]
Combining this with the formula for the difference $T(\{ \pi^{[a,b]} \}) - B(\{ \pi^{[a,b]} \})$ (Proposition \ref{difference}) gives
\begin{eqnarray*}
 N^\sigma(y) &=& \sum_{\{ \pi^{[a,b]} \}}\left( \prod_{[a,b]\not\in S} f\left( \pi^{[a,b]} \right) \right) \cdot \left( \prod_{[a,b]\in S} f\left( \pi^{[a,b]}\right) \prod_{l\geq 1} (-\mathbb{L}^{\half})^{-(b_l^{a,b})^2} \right) \\
& &
\hspace{ 2.0cm} \cdot \prod_{i=0}^{N-1} y_i^{\sum_{0\leq a,b \leq N-1}|\pi^{[a,b]}| - \sum_{[a,b]\not\ni i} l(\pi^{[a,b]})} . \end{eqnarray*}
To simplify notation set
\[\begin{array}{lllll} g\left(\pi\right) &:=& f\left( \pi \right) \cdot \prod_{l\geq 1} (-\mathbb{L}^{\half})^{-b_l^2} &\textrm{ for }& \pi = (1^{b_1}2^{b_2}3^{b_3}\cdots) \\
 &&  \end{array}\]
then rearranging the products and summations gives
\begin{eqnarray*}
 N^\sigma(y) &=& \prod_{[a,b]\not\in S} \sum_{\pi^{[a,b]}} f\left( \pi^{[a,b]} \right) \cdot y'^{\left|\pi^{[a,b]}\right|-l\left(\pi^{[a,b]}\right)} \cdot y_{[a,b]}^{l(\pi^{[a,b]})} \\ 
 & & \cdot \prod_{[a,b]\in S} \sum_{\pi^{[a,b]}} g\left(\pi^{[a,b]}\right) \cdot  y'^{\left|\pi^{[a,b]}\right|-l\left(\pi^{[a,b]}\right)}\cdot y_{[a,b]}^{l(\pi^{[a,b]})}. 
\end{eqnarray*}
 Both of these series are know to have product expansions \cite{mac}
\[ \begin{array}{lllll}
 f(t,a) &=& \sum_{\pi} f\left(\pi\right)a^{l(\pi)} t^{|\pi|-l(\pi)} &=& \Exp\left( \frac{1}{1-\mathbb{L}^{-1}}\cdot \frac{a}{1-t} \right) \\
 g(t,a) &=& \sum_{\pi} g\left(\pi\right)a^{l(\pi)} t^{|\pi|-l(\pi)} &=&  \Exp\left( \frac{(-\mathbb{L}^{\half})^{-1}}{1-\mathbb{L}^{-1}} \cdot \frac{a}{1-t} \right).
\end{array}\]
Now $N^\sigma$ is a product of such series and multiplying together the corresponding exponential generating series gives the desired result
\[N^\sigma(y) = \Exp \left( \frac{\mathbb{L}}{\mathbb{L}-1}\frac{1}{1-y'}\left(
\sum_{[a,b] \not\in S} y _{[a,b]} -\mathbb{L}^{-\half}\sum_{[a,b] \in S} y_{[a,b]}\right) \right).\]
\end{proof}
 Now we have computed $I^\sigma$, $N^\sigma$ and so by lemma \ref{factor}
\[ A^\sigma_U(y) = \Exp\left(  \mathbb{L} \frac{y'}{1-y'} + \frac{\mathbb{L}}{\mathbb{L}-1}\frac{1}{1-y'}\left(
\sum_{[a,b] \not\in S} y _{[a,b]} -\mathbb{L}^{-1/2}\sum_{[a,b] \in S} y_{[a,b]}\right) \right) \]

or to reformulate this as a product over the set of roots
\[
\Exp\left( \frac{ 1}{1-\mathbb{L}^{-1}}\left( (\mathbb{L} + N-1)\sum_{\alpha\in\Delta_{\sigma,+}^{im} }y^{\alpha} +  
\sum_{\substack{\alpha\in\Delta_{\sigma,+}^{re}\\ \sum_{I_2\cup I_3} \alpha_i \textrm{ is even}}} y^{\alpha} 
-\mathbb{L}^{-\half}\sum_{\substack{\alpha\in\Delta_{\sigma,+}^{re} \\ \sum_{I_2\cup I_3} \alpha_i \textrm{ is odd}}}y^\alpha
 \right)\right).  \]
 Thus proving Theorem \ref{thm_A}
 \[ A^\sigma_U(y) = \prod_{\alpha\in \Delta_{\sigma,+}}A^\alpha(y) \]
 for the special case of the partition $\sigma$.

%% file: universal_small_2.tex
\section{The universal DT series : general case}\label{sec_univ2}
In this section we will prove Theorem \ref{thm_A} for any partition $\sigma$.

\subsection{Mutation and the root system}\label{subsec_root}
%By the theory of root system, the simple reflection provides a bijection between $\Delta_{\sigma,+}\backslash \{\alpha_k\}$ and $\Delta_{\sigma',+}\backslash \{\alpha'_k\}$. 
%The simple root $\alpha_k$ maps to $-\alpha'_k$.
Recall that the simple reflection provides a bijection between $\Delta_{\sigma,+}\backslash \{\alpha_k\}$ and $\Delta_{\sigma',+}\backslash \{\alpha'_k\}$ (see \S \ref{subsec_mutation}). 
The simple root $\alpha_k$ maps to $-\alpha'_k$.

For $\alpha\in \Delta^{\mr{re}}_+$, let $x_\alpha$ be a simple module with $\underline{\mathrm{dim}}\alpha$.
By \cite[Proposition 2.14]{3tcy}, 
$\sum_{i\notin \hat{I}_r}\alpha_i$ is odd (resp. even) if and only if 
$\mathrm{ext}^1(x,x)=0$ (resp. $=1$).
In particular, the parity of $\sum_{i\notin \hat{I}_r}\alpha_i$ is preserved by the simple reflection.

\subsection{Wall-crossing formula}\label{subsec_univ2}
\begin{thm}\cite[Theorem 4.9]{motivic_WC}\label{thm_WC}
\[
A_U^{\sigma'}(\mathbf{y})
=
\frac{A_U^{\sigma}(\mathbf{y})}{\mathbb{E}(y_k)}\times \mathbb{E}(y_k^{-1})
\]
\end{thm}
\begin{proof}

\noindent\underline{Step 1} :
By the observation in \S \ref{subsec_mutation}, we have the following factorization:
\[
A_U^\sigma=\mathbb{E}(y_k)\times A_k^\sigma
\]
where 
%\begin{align*}
\[
\mathbb{E}(y):=\sum_{n\geq 0}\frac{[\mathrm{pf}]}{[\mathrm{GL}_n]_\mr{vir}}\cdot y^n,\quad y_k:=y_{\alpha_k}
\]
%\end{align*}
and $A_k^\sigma$ is the generating series of virtual motives of moduli stacks of objects in $(\mathrm{mod}J_{\sigma})_k$. 
We also have
\[
A_U^{\sigma'}=A^{\sigma',k} \times \mathbb{E}(y_k^{-1})
\]
where $A^{\sigma',k}$ is the generating series of virtual motives of moduli stacks of objects in $(\mathrm{mod}J_{\sigma'})^k$. 

\noindent\underline{Step 2} : 
By Proposition \ref{prop_cut_mutation}, we have $A^{\sigma}_k=A^{\sigma',k}$ (see \cite[Theorem4.7]{motivic_WC}).
\end{proof}
%\begin{rk}
%In \cite{motivic_WC}, we assume that the quiver has no loops and $2$-cycles so that mutation of quivers with potential in the sense of \cite{quiver-with-potentials} is well-defined. 
%In our setting, although the relation between $J_\sigma$ and $J_{\sigma'}$ is not included in the general theory of \cite{quiver-with-potentials}, we have all we need to apply the arguments in \cite{motivic_WC}, that is, 
%\begin{itemize}
%\item
%the derived equivalence, and
%\item
%the description of the change of t-structures in \S \ref{}.
%\end{itemize}
%\end{rk}
%\begin{thm}
%\label{thm main universal result}
%For any partition $\sigma$, we have
%\[
%A_U^\sigma(\mathbf{y})=
%\prod_{\al\in\De_{\sigma,+}}
%A^{\al}(\mathbf{y}).
%\]
%\label{eq univ exp form}
%\end{thm}
%\begin{proof}
Now Theorem \ref{thm_A} for any $\sigma$ follows from the result in \S \ref{sec_univ1} combined with Theorem \ref{thm_WC} and the remark in \S \ref{subsec_root}.
%\end{proof}

\subsection{Factorization of the universal series}\label{subsec_factorization}
We will say that a stability parameter $\ze$ is generic, if for any stable $J_\sigma$-module $V$, we have $\ze\cdot\udim V\ne0$.
For generic stability parameter $\ze$, let $\gM_\ze^+(J_\sigma,\al)$ (resp.\ $\gM_\ze^-(J_\sigma,\al)$) denote the moduli stacks 
of $J_\sigma$-modules $V$ such that $\udim V=\al$ and such that all the HN factors $F$ of $V$ with respect to the stability parameter \ze satisfy $\ze\cdot\udim F>0$ (resp. $<0$).
Let $[\gM_\ze^\pm(J_\sigma,\al)]_\vir$ denote the virtual motive of the moduli stack defined in the same way as \eqref{eqref_33}.
We put
\[
A_\ze^\pm(y)=\sum_{\al\in\mathbb{N}^{\hI}}[\gM_\ze^\pm(J,\al)]_\vir\cdot y^\al.
\] 
\begin{lem}\cite[Lemma 2.6]{mmns}\label{lem_52}
\label{lmm:decompose}
The generating series $A_\ze^\pm$ are given by 
\[
A_\ze^\pm(y)=
\prod_{\al\in\De_{\sigma,+}, \pm\ze\cdot\alpha<0}
A^{\al}(y).
\]
\end{lem}

%% file: framing.tex
\section{Motivic DT with framing}\label{sec_framing}
\label{sec:framing}
%\subsection{Motivic DT invariants with framing}
We denote by $\tilde{Q}_\sigma$ the new quiver obtained from $Q_\sigma$ by adding a new vertex $\infty$ and a single new arrow $\infty\to 0$. 
Let $\tilde{J}_\sigma=J_{\tilde{Q}_\sigma,w_\sigma}$ be the Jacobian algebra corresponding to the quiver with potential $(\tilde{Q}_\sigma,w_\sigma)$, where we view $w_\sigma$ as a potential for $\tilde{Q}_\sigma$ in the obvious way. 
%Any $\wtl Q$-representation (resp.\ $\wtl J$-module) $\wtl V$ can be written as a triple $(V,\wtl V_\infty,s)$, where $V$ is a $Q$-representation (resp.\ $J$-module), $\wtl V_\infty$ is a vector space, and $s:\wtl V_\infty\to V_0$ is a linear map. We will always do this identification without mentioning.

%The twisted motivic algebra $\That_Q$ of the original quiver sits as a subalgebra inside the algebra $\That_{\Qtilde}$ associated to the framed quiver $\Qtilde$. Note that in $\That_{\Qtilde}$ we have
%\begin{equation}
%y_\infty\cdot y^{(\al,0)}=(-\cL^\oh)^{-\al_0}\cdot y^{(\al,1)}=\cL^{-\al_0} \cdot y^{(\al,0)}\cdot y_\infty,
%\label{eq:y infty rel}
%\end{equation}
%where we put
%\[y_\infty=y^{(0,1)}.\]
%In particular, $\That_{\Qtilde}$ is never commutative.

%\subsection{Stability for framed representations}
%\begin{figure}[ht]
%  \centering
%  \input{pic1.tpc}
%  \caption{The quiver $Q$}
%  \label{fig:Q}
%\end{figure}

Let $\ze\in\cR^{\hI}$ be a vector, which we will refer to as the stability parameter.
%\begin{defn} 
A $\tilde{J}_\sigma$-representation $\wtl V$ with $\dim\wtl V_\infty= 1$ is said to be $\ze$-(semi)stable, if it is (semi)stable with respect to $(\ze,\ze_\infty)\in\mathbb{R}^{\hI\sqcup \{\infty\}}$ (see Definition \ref{semist}), where $\ze_\infty=-\ze\cdot\udim V$. 
As in \S \ref{subsec_32}, a stability parameter $\ze\in\cR^{Q_0}$ is said to be {\it generic}, if for any stable $J$-module $V$ we have $\ze\cdot\udim V\ne0$.
%if $\ze$-stability and $\ze$-semi\-sta\-bi\-lity of $\wtl Q$-modules coincide for all dimension vectors.

For a stability parameter $\ze\in\cR^{Q_0}$ and a dimension vector $\al\in (\mathbb{Z}_{\geq 0})^{\hI}$, 
let $\gM_\ze(\wtl J_\sigma,\al)$ denote the moduli stack of $\ze$-semistable $\wtl J_\sigma$-representations  with dimension vector $(\alpha,1)$. 
As in the introduction, we define the generating function:
\[
Z_\ze(y_0,\ldots,y_{N-1})=
Z_\ze(y):=
\sum_{\al\in(\mathbb{Z}_{\geq 0})^{\hI}}
\Bigl[{\MM}_\ze\bigl(\Jtilde_\sigma,\al\bigr)\Bigr]_\vir\cdot y^\al.
\]
\begin{thm}\cite[Proposition 4.6]{mmns} 
\label{thm_framed_vs_nonframed}
For a generic stability parameter \ze, we have
\begin{equation}
Z_\ze(y)
=\frac{A_\ze^-(-\cL^{\oh}y_0,y_1,\ldots, y_{N-1})}{A_\ze^-(-\cL^{-\oh}y_0,y_1,\ldots, y_{N-1})},
\end{equation}
where $A_\ze^-$ were defined in \S \ref{subsec_factorization}.
\end{thm}

%\begin{thm}
%For $\ze\in\R^{\tI}$ not orthogonal to any root, we have
%\[
%Z_\ze(\mathbf{y})=
%\prod_{\substack{\al\in\De_{\sigma,+}\\ \ze\cdot\al<0}}
%Z_{\al}(\mathbf{y}).
%\]
%\end{thm}
Combined with Lemma \ref{lem_52}, we get the formula in Corollary \ref{cor_02}.
\begin{rk}
If we cross the wall $W_\alpha$, we get (or lose) a factor $Z_{\al}(y)$ in the generating function. 
This is compatible with the result in \cite{open_3tcy}.
\end{rk}

%% file: DTPT.tex
\section{DT/PT series}\label{sec_DTPT}
\subsection{Chambers in the moduli spaces}
For a root $\alpha\in \Lambda$, 
let $W_\alpha$ denote the hyperplane in the space $\mathbb{R}^{\hI}$ of stability parameters which is orthogonal to $\alpha$. 
We put 
\[
W=W_\delta\cup \bigcup_{\alpha\in \Delta_{\sigma,+}^{\mr{re}}}W_\alpha. 
\]
A connected component of the complement of $W$ in $\mathbb{R}^{\hI}$ is called a chamber.
\begin{thm} {\rm\cite[Proposition 2.10]{3tcy}, \cite[Proposition 3.10, 3.11]{nagao-nakajima}} \label{thm_desrc}
The set of generic parameters in $\mathbb{R}^{\hI}$ is the compliment of $W$.
\begin{enumerate}\renewcommand{\theenumi}{\roman{enumi}}
\item For $\ze$ with $\zeta_i<0$ ($\forall i$), 
the moduli spaces $\MM_\ze(\Jtilde,\al)$ are the NCDT moduli spaces, the moduli spaces of cyclic $J$-modules from~\cite{szendroi-ncdt}.
\item 
For $\ze$ in the same chamber as $(1-N+\varepsilon,1,1,\ldots,1)$ ($0<\varepsilon\ll 1$), the moduli spaces $\MM_\ze(\Jtilde,\al)$ are the DT moduli spaces of $Y_\sigma$ from~\cite{mnop}, 
the moduli spaces of subschemes on $Y_\sigma$ with support in dimension at most 1.
\item 
For $\ze$ in the same chamber as $(1-N-\varepsilon,1,1,\ldots,1)$ ($0<\varepsilon\ll 1$), 
the moduli spaces $\MM_\ze(\Jtilde, \al)$ are the PT moduli spaces of $Y_\sigma$ introduced in~\cite{pt1}; 
these are moduli spaces of stable rank-$1$ coherent systems. 
\end{enumerate}
\end{thm}
\begin{rk} In the above statements $\varepsilon$ depends on the dimension vector $(\al,1)$. 
\end{rk}

\subsection{Motivic PT and DT invariants}

Let 
\[
\zeta_{\mr{DT}}=(1-N-\varepsilon,1,1,\ldots,1)
,\quad
\zeta_{\mr{PT}}=(1-N+\varepsilon,1,1,\ldots,1)
\] 
($0<\varepsilon\ll 1$) be stability parameters corresponding to DT and PT moduli spaces. 
Then we have
\begin{align*}
\{\alpha\in \Delta_{\sigma,+} \mid \ze_{\mathrm{DT}}\cdot\al<0\}
&=\Delta^{\mr{re},+}_+,\\
\{\alpha\in \Delta_{\sigma,+} \mid \ze_{\mathrm{PT}}\cdot\al<0\}
&=\Delta^{\mr{re},+}_+\sqcup \Delta^{\mr{im}}_+.
\end{align*}
As we mentioned in the introduction the variable change induced by the derived equivalence is given by
\[
s:=y_0\cdot y_1\cdot\cdots\cdot y_{N-1},\quad 
%t_{[a,b]}=y_1\cdot\cdots\cdot y_{b}
T_i=y_i.
\]
Here $s$ is the variable for the homology class of a point and 
$T_i$ is the variable for the homology class of $C_i$.
Then we get the formulae in Corollary \ref{announce_thm_main_cor}.

\subsection{Connection with the refined topological vertex}
As the second author studied in \cite{ncdt-vo}, we can apply the vertex operator method \cite{ORV} to get a product expansion of the refined topological vertex for $\mathcal{Y}_\sigma$.
Then we see that the PT generating function can be described by the refined topological vertices normalized by the refined MacMahon functions. \footnote{Unfortunately, the DT generating function does not coincide with the refined topological vertex. See \cite[\S 4.3]{mmns} for detail.}

%% file: App.tex
\thispagestyle{empty}

\section{Appendix}
Throughout this appendix we will work with a fixed choice of basis $\mathcal{B}$. In \S \ref{nil} we chose a basis
\[ \mathcal{B }  =  \{ H^{p}v^{a,b}_l(k) \mid  1\leq k \leq b^{a,b}_l,  0\leq a,b \leq N-1, 0 \leq p \leq N(l-1)+ |b-a| +1 \} \]
and defined linear spaces
\begin{eqnarray*}
F(\{ \pi^{[a,b]} \}) & = & \left\{ L\in \bigoplus_{i\in I_1} \Hom(V_i,V_i) \oplus\bigoplus_{i\in I_2} \Hom(V_{i+1},V_i)\, \Big|\, (L,H(\pi^{[a,b]})) \in R(\{ \pi^{[a,b]} \}) \right\} \\
S(\{ \pi^{[a,b]} \}) & = & \left\{ N \in \bigoplus_{i\in I_1\cup I_2 \cup I_3} \Hom(V_i,V_i) \, \Big|\, \left[ N, H(\pi^{[a,b]}) \right] =0 \right\}
\end{eqnarray*}
with dimensions $ T(\{ \pi^{[a,b]} \})= \dim F(\{ \pi^{[a,b]} \})$ and $ B(\{ \pi^{[a,b]} \}) = \dim S(\{ \pi^{[a,b]} \})$. The goal of the appendix is to prove Proposition \ref{difference}, that is to show that the difference $ T(\{\pi^{[a,b]}\}) - B(\{\pi^{[a,b]}\}) $ is equal to
\[-\frac{1}{2}\sum_{i\in I_2} \left( \sum_{ b \neq i } l(\pi^{[i+1,b]}) - \sum_{ c\neq i+1 }l(\pi^{[c,i]}) \right)^2 
- \half \sum_{a\in I_3 , b\not\in I_2} \sum_{i\geq 1}(b_i^{a,b})^2 - \half\sum_{a\not\in I_3 , b\in I_2}  \sum_{i\geq 1} (b_i^{a,b})^2. \] For some early examples it becomes clear that the dimension of $F(\{ \pi^{[a,b]} \})$ and $S(\{ \pi^{[a,b]} \})$ are determined by solving a set of linearly independent equations. We will see that these dimensions are quadratic polynomials in the number of parts $b^{a,b}_l$ of the partitions $\{ \pi^{[a,b]} \}$. An initial means of simplifying the calculation is to break the spaces $F(\{ \pi^{[a,b]} \})$ and $S(\{ \pi^{[a,b]} \})$ down into simpler spaces. One easy observation is that not only are the spaces $F(\{ \pi^{[a,b]} \})$ and $S(\{ \pi^{[a,b]} \})$ linear but they come with a natural vector space structure, the origin corresponding to the zero matrix in both cases.  This means that as vector spaces we have decompositions 
\begin{eqnarray*}
F(\{ \pi^{[a,b]} \}) &=& \bigoplus_{0\leq a,b,c,d \leq N-1} F(\pi^{[a,b]},\pi^{[c,d]}) \\
S(\{ \pi^{[a,b]} \}) &=& \bigoplus_{0\leq a,b,c,d \leq N-1} S(\pi^{[a,b]},\pi^{[c,d]})
\end{eqnarray*}
whose summands are given by the following definition.
\begin{defn}We define
\begin{eqnarray*}
F(\pi^{[a,b]},\pi^{[c,d]}) &=& F(\{ \pi^{[a,b]} \}) \cap \bigoplus_{i\in I_1\cup I_2} \Hom(V^{a,b} ,V^{c,d}) \\
S(\pi^{[a,b]},\pi^{[c,d]}) &=& S(\pi^{[a,b]},\pi^{[c,d]}) \cap  \bigoplus_{i\in I_1\cup I_2 \cup I_3} \Hom(V^{a,b} ,V^{c,d}).
\end{eqnarray*}
\end{defn}
These subspaces are essentially given by the block matrices for the decomposition $V = \bigoplus_{0\leq a,b \leq N-1} V^{a,b}$.
\begin{defn}We define
\begin{eqnarray*}
T(\pi^{[a,b]},\pi^{[c,d]}) &=& \dim F(\pi^{[a,b]},\pi^{[c,d]}) \\
B(\pi^{[a,b]},\pi^{[c,d]}) &=& \dim S(\pi^{[a,b]},\pi^{[c,d]}) .
\end{eqnarray*}
\end{defn}
Both $T(\pi^{[a,b]},\pi^{[c,d]})$ and $B(\pi^{[a,b]},\pi^{[c,d]})$ can be written as quadratic expressions in the number of parts of $\pi^{[a,b]}$ and $\pi^{[c,d]}$. To do this we introduce a quadratic form on the space of all partitions and a combinatorial operation that removes a box from each column of the the partition.
\begin{defn}We define
\begin{eqnarray*} M &:& \mathcal{P} \otimes \mathcal{P} \to \mathbb{Z}_{\geq 0} \\
 &:& (1^{b_1}2^{b_2}3^{b_3}\cdots )\otimes (1^{c_1}2^{c_2}3^{c_3}\cdots) \mapsto \sum_{i\geq 1} \left(\sum_{j\geq i } b_j \right)\left(\sum_{j\geq i } c_j \right) \\
 ' &:& \mathcal{P} \to \mathcal{P} \\ &:& \pi = (1^{b_1}2^{b_2}3^{b_3}\cdots) \mapsto \pi '= (1^{b_2}2^{b_3}3^{b_4}\cdots). 
\end{eqnarray*}
\end{defn}
Let us begin with the easier case. We compute dimensions $B(\pi^{[a,b]},\pi^{[c,d]}) $ of the spaces $S(\pi^{[a,b]},\pi^{[c,d]})$.
\begin{lem}\label{Ndeter}Let $N\in S(\pi^{[a,b]},\pi^{[c,d]}) $ then the matrix $N$ is uniquely determined by its value on the vectors $v^{a,b}_l(k)$. Moreover the only restriction on the image of such a vector is that it lie in the linear subspace
\[ N(v^{a,b}_l) \in V_a\cap V^{c,d}\cap V^{N\cdot(l-1) + |b-a| +1}. \]
\end{lem}
\begin{proof}To define the linear map $N$ on the space $V^{a,b}$ it suffices to define its value at each of the basis vectors
\[ \{ H^rv_l^{a,b}(k) \mid 0 \leq r \leq N\cdot (l-1) +|b-a| , 1 \leq k \leq b_l^{a,b}  \}. \]
However for $N\in S( \pi^{[a,b]},\pi^{[c,d]})$ we have 
\[ N(H^rv_l^{a,b}(k)) = H^r(Nv_l^{a,b}(k)), \]
therefore the value of $N$ at each $H^r v_l^{a,b}(k)$ is determined by $N v_l^{a,b}(k)$. This proves the first part of the lemma. Now we know that the matrix $N$ maps the vector space at the $a$th vertex to itself $V_a\to V_a$, also since $N\in S( \pi^{[a,b]},\pi^{[c,d]})$ we insist that its image be in $V^{c,d}$. The only additional condition on the image of the vector $v^{a,b}_l(k)$ is
\[ H^{N\cdot (l-1) + |b-a| +1}(Nv_l^{a,b}(k)) = N(H^{N\cdot (l-1) + |b-a| +1}v_l^{a,b}(k)) = 0. \]
Combining these three conditions above we have,
\[ N(v_l^{a,b}(k)) \in V_a\cap V^{c,d} \cap V^{N\cdot (l-1) + |b-a|+1}. \]\end{proof}
\begin{cor}\label{Bdim}We have
\[ B(\pi^{[a,b]},\pi^{[c,d]}) = \left\{ \begin{array}{ll} M(\pi^{[a,b]},\pi^{[c,d]}) &  \textrm{ if } a\in [c,d] \textrm{ and } |d-a| \leq |b-a|  \\
M((\pi^{[a,b]})',\pi^{[c,d]}) & \textrm{ if }  a \in [c,d] \textrm{ and } |d-a| > |b-a|  \\
M(\pi^{[a,b]} ,(\pi^{[c,d]})') & \textrm{ if } a\not\in [c,d] \textrm{ and } |d-a| \leq |b-a|  \\
M((\pi^{[a,b]}) ',(\pi^{[c,d] })') & \textrm{ if } a\not\in [c,d] \textrm{ and } |d-a| > |b-a|. \end{array} \right .\]
\end{cor}
\begin{proof}Let $N\in S(\pi^{[a,b]},\pi^{[c,d]})$. Each vector $v^{a,b}_l(k)$ with $1\leq k \leq b_l^{a,b}$ can take any value in the vector space $V_a\cap V^{c,d} \cap V^{N\cdot (l-1) + |b-a|+1}$ and so the dimension of $S(\pi^{a,b},\pi^{[c,d]})$ is given by 
\[ B(\pi^{[a,b]},\pi^{[c,d]})= \sum_{l\geq 0} b_l^{a,b}\cdot \dim \left( V_a\cap V^{c,d} \cap V^{N\cdot (l-1) +|b-a|+1} \right). \]
Counting the number of basis vectors of $V^{c,d}$ that lie in $V_a$ we see there are four possibilities for $\dim \left( V_a\cap V^{c,d} \cap V^{N\cdot (l-1) +|b-a|+1} \right)$:
\[  \left. \begin{array}{ll} 
\sum_{i=1}^{l} ib_i^{c,d} + l\sum_{i\geq l} b_i^{c,d}&\textrm{ if } a\in [c,d] \textrm{ and } |d-a| \leq |b-a|  \\
\sum_{i=1}^{l-1} ib_i^{c,d} + (l-1)\sum_{i\geq l} b_i^{c,d} & \textrm{ if }  a \in [c,d] \textrm{ and } |d-a| > |b-a|  \\
\sum_{i=1}^{l} ib_{i+1}^{c,d} + l\sum_{i\geq l} b_{i+1}^{c,d} & \textrm{ if } a\not\in [c,d] \textrm{ and } |d-a| \leq |b-a|  \\
\sum_{i=1}^{l-1} ib_{i+1}^{c,d} + (l-1)\sum_{i\geq l} b_{i+1}^{c,d}  & \textrm{ if } a\not\in [c,d] \textrm{ and } |d-a| > |b-a|. 
\end{array} \right .\]
Consider the first case $a\in [c,d]$ and $|d-a| \leq |b-a|$ then
\begin{eqnarray*}
B(\pi^{[a,b]},\pi^{[c,d]}) & = & \sum_{l\geq 1} b_l^{a,b} \cdot  \left( \sum_{i=1}^l ib_i^{c,d} + l\sum_{i\geq l} b_i^{c,d} \right) \\ 
& = & \sum_{i\geq 1} \left( \sum_{l\geq i} b^{a,b}_l \right) \cdot \left( \sum_{l\geq i} b^{c,d}_l \right) \\
 &=& M(\pi^{[a,b]},\pi^{[c,d]}).
\end{eqnarray*}
The other three cases are identical. The relabeling of the partitions in these cases is encoded by the operation $\pi \mapsto \pi '$.
\end{proof}
Now we turn to computing the dimensions $T(\pi^{[a,b]},\pi^{[c,d]})$ of the spaces $F(\pi^{[a,b]},\pi^{[c,d]})$. This will be more intricate. 
\begin{lem}\label{LimI_1I_3}
Suppose $a\in I_1\cup I_3$ and $L\in F(\pi^{[a,b]},\pi^{[c,d]})$ then the map $L$ is uniquely determined by its value on the vectors $v^{a,b}_l(k)$. Moreover the only restriction on the image of such a vector is that it lie in a linear subspace
\[ Lv^{a,b}_l(k) \in \left\{ \begin{array}{ll} 
V_a\cap V^{N\cdot (l-1) + |b-a| +1}\cap V^{c,d} &\textrm{ if } a\in I_1 \textrm{ and }b\not\in I_2 \\
V_a\cap V^{N\cdot (l-1) + |b-a| }\cap V^{c,d} & \textrm{ if } a\in I_1 \textrm{ and } b\in I_2 \\
V_{a-1}\cap V^{N\cdot (l-1) + |b-a| +2}\cap V^{c,d} &\textrm{ if } a\in I_3 \textrm{ and }b\not\in I_2 \\
V_{a-1}\cap V^{N\cdot (l-1) + |b-a| +1}\cap V^{c,d}& \textrm{ if } a\in I_3 \textrm{ and } b\in I_2 .
\end{array}\right. \] 
\end{lem}
\begin{proof}To define the linear map $L$ on the space $V^{a,b}$ it suffices to define its value at each of the basis vectors
\[ \{ H^rv_l^{a,b}(k) \mid 0 \leq r \leq N\cdot (l-1) +|b-a| , 1 \leq k \leq b_l^{a,b}  \}. \]
However for $L\in F( \pi^{[a,b]},\pi^{[c,d]})$ we know that the pair $(L,H(\pi^{[a,b]}))\in R(\{\pi^{[a,b]}\})$ satisfy the relations coming from the superpotential:
\[\begin{array}{llll}
r_i h^+_{i-1/2} &=& h_{i-1/2}^+r_{i-1} &\textrm{ for } i\in [1,N'-1]\cap I_1 \\
r_0h^+_{N-1/2} &=&h^+_{N-1/2} h^+_{N-3/2}  h^-_{N-3/2} &\\
h^-_{N'+1/2} h^+_{N'+1/2}h^+_{N'-1/2}&=& h^+_{N'-1/2}r_{N'-1} & \\
h^-_{i+3/2} h^+_{i+3/2}h^+_{i+1/2}&=&h^+_{i+1/2} h^+_{i-1/2}  h^-_{i-1/2}& \textrm{ for } i=[N'+1,N-3 ]\cap I_3.
\end{array}\]
As in Lemma \ref{Ndeter} once the value of $L$ is determined for $v^{a,b}_l(k)$ it is uniquely determined  for all $H^r v^{a,b}_l(k)$ by the condition that the above relations be satisfied for the pair $(L,H(\pi^{[a,b]})$. To be precise if $a\in I_1$ we have
\[  L: H^{r}(v^{a,b}_l(k)) \mapsto 
\left\{ \begin{array}{ll} 
H^{r}L (v^{a,b}_l(k))& \textrm{ if } a+r \in I_1 \\  
0 & \textrm{ if } a+r \in I_2 \\
H^{r-1}L (v^{a,b}_l(k)) & \textrm{ if } a+r \in I_3 
\end{array} \right. \]
and if  $a\in I_3$ then
\[  L: H^{r}(v^{a,b}_l(k)) \mapsto 
\left\{ \begin{array}{ll} 
H^{r+1}L (v^{a,b}_l(k))& \textrm{ if } a+r \in I_1 \\  
0 & \textrm{ if } a+r \in I_2 \\
H^{r}L (v^{a,b}_l(k)) & \textrm{ if } a+r \in I_3 .
\end{array} \right. \]
Since $L\in F(\pi^{[a,b]},\pi^{[c,d]})$ by definition its image must lie in the space $V^{c,d}$, also if $a\in I_1$ then $L:V_a\to V_a$ and if $a\in I_3$ then $L:V_a\to V_{a-1}$. The only further condition on the image of a vector $v^{a,b}_l(k)$ is that its image be killed by a high enough power of $H$. It is given that $H^{N\cdot (l-1) + |b-a|+1}v^{a,b}_l(k)=0$ so then $H^t(Lv^{a,b}_l(k))=0$ where the exponent $t$ is read off for the defining relations on $L$ above. In the separate cases
\[ Lv^{a,b}_l(k) \in \left\{ \begin{array}{ll} 
V_a\cap V^{N\cdot (l-1) + |b-a| +1}\cap V^{c,d} &\textrm{ if } a\in I_1 \textrm{ and }b\not\in I_2 \\
V_a\cap V^{N\cdot (l-1) + |b-a| }\cap V^{c,d} & \textrm{ if } a\in I_1 \textrm{ and } b\in I_2 \\
V_{a-1}\cap V^{N\cdot (l-1) + |b-a| +2}\cap V^{c,d} &\textrm{ if } a\in I_3 \textrm{ and }b\not\in I_2 \\
V_{a-1}\cap V^{N\cdot (l-1) + |b-a| +1}\cap V^{c,d}& \textrm{ if } a\in I_3 \textrm{ and } b\in I_2 .
\end{array}\right. \] 
proving the result.
 \end{proof}

We have a result similar to Lemma \ref{LimI_1I_3} when $a\in I_2$.
\begin{lem}\label{LimI_2}
Suppose $a\in I_2$ and $L\in F(\pi^{[a,b]},\pi^{[c,d]})$ then the map $L$ is uniquely determined by its value on the vectors $Hv^{a,b}_l(k)$. Moreover the only restriction on the image of such a vector is that it lie in a linear subspace
\[ L(Hv^{a,b}_l(k)) \in \left\{ \begin{array}{ll} 
V_a\cap V^{N\cdot (l-1) + |b-a| +1}\cap V^{c,d} &\textrm{if } b\not\in I_2 \\
V_a\cap V^{N\cdot (l-1) + |b-a| }\cap V^{c,d} &\textrm{if } b\in I_2 
\end{array}\right. \] 
\end{lem}
\begin{proof}Again we know that to define the linear map $L$ on the space $V^{a,b}$ it suffices to define its value at each of the basis vectors
\[ \{ H^rv_l^{a,b}(k) \mid 0 \leq r \leq N\cdot (l-1) +|b-a| , 1 \leq k \leq b_l^{a,b}  \}. \]
Since by definition if $a\in I_2$ then $Lv^{a,b}_l(k)=0$ the map is already trivially determined on these vectors and their image does not suffice to determine the map in general. However if we consider the vectors $Hv^{a,b}_l(k)$ then once the value of $L$ is determined for $Hv^{a,b}_l(k)$ it is uniquely determined  for all $H^r v^{a,b}_l(k)$ by the condition that the relations (see Lemma \ref{LimI_1I_3}) be satisfied by the pair $(L,H(\pi^{[a,b]})$. To be precise if $a\in I_2$ we have
\[  L: H^{r}(v^{a,b}_l(k)) \mapsto 
\left\{ \begin{array}{ll} 
H^{r}L (Hv^{a,b}_l(k))& \textrm{ if } a+r \in I_1 \\  
0 & \textrm{ if } a+r \in I_2 \\
H^{r-1}L(Hv^{a,b}_l(k)) & \textrm{ if } a+r \in I_3 .
\end{array} \right. \] 
By definition we know that the image of $L$ lies in $V^{c,d}$ and also that for $a\in I_2$ we have $L:V_{a+1}\to V_a$. As before the only remaining condition on the image of $v^{a,b}_l(k)$ is that it be killed by a high enough power of $H$. From the definition of $L$ above we see that
\[ L(Hv^{a,b}_l(k)) \in \left\{ \begin{array}{ll} 
V_a\cap V^{N\cdot (l-1) + |b-a| +1}\cap V^{c,d} &\textrm{if } b\not\in I_2 \\
V_a\cap V^{N\cdot (l-1) + |b-a| }\cap V^{c,d} &\textrm{if } b\in I_2 
\end{array}\right. \] 
proving the result.
\end{proof}
The following notation collects the dimensions of all the vector spaces encountered in the last two Lemmas.
\begin{defn} We define integers
\[ d_{a,b:c,d}(l) =\left\{ \begin{array}{ll} 
\dim( V_a\cap V^{N\cdot (l-1) + |b-a| +1}\cap V^{c,d} )&\textrm{ if } a\in I_1\cup I_2 \textrm{ and }b\not\in I_2 \\
\dim( V_a\cap V^{N\cdot (l-1) + |b-a| }\cap V^{c,d} ) & \textrm{ if } a\in I_1\cup I_2 \textrm{ and } b\in I_2 \\
\dim( V_{a-1}\cap V^{N\cdot (l-1) + |b-a| +2}\cap V^{c,d} )&\textrm{ if } a\in I_3 \textrm{ and }b\not\in I_2 \\
\dim( V_{a-1}\cap V^{N\cdot (l-1) + |b-a| +1}\cap V^{c,d} ) & \textrm{ if } a\in I_3 \textrm{ and } b\in I_2.
\end{array}\right. \]
\end{defn}
From Lemmas \ref{LimI_1I_3} and \ref{LimI_2} we deduce the dimension of the spaces $F(\pi^{[a,b]},\pi^{[c,d]})$.
\begin{cor}\label{Tdim}If $a\in I_1\cup I_2$ and $b \not\in I_2$ then 
\[ T(\pi^{[a,b]},\pi^{[c,d]}) = \left\{ \begin{array}{ll} M(\pi^{[a,b]},\pi^{[c,d]}) &  \textrm{ if } a\in [c,d] \textrm{ and } |d-a| \leq |b-a|  \\
M((\pi^{[a,b]})',\pi^{[c,d]}) & \textrm{ if }  a \in [c,d] \textrm{ and } |d-a| > |b-a|  \\
M(\pi^{[a,b]} ,(\pi^{[c,d]})') & \textrm{ if } a\not\in [c,d] \textrm{ and } |d-a| \leq |b-a|  \\
M((\pi^{[a,b]}) ',(\pi^{[c,d] })') & \textrm{ if } a\not\in [c,d] \textrm{ and } |d-a| > |b-a|. \end{array} \right .  \]
If $a\in I_1\cup I_2$ and $b \in I_2$ then
\[ T(\pi^{[a,b]},\pi^{[c,d]}) = \left\{ \begin{array}{ll} M(\pi^{[a,b]},\pi^{[c,d]}) &  \textrm{ if } a\in [c,d] \textrm{ and } |d-a| \leq |b-a| -1 \\
M((\pi^{[a,b]})',\pi^{[c,d]}) & \textrm{ if }  a \in [c,d] \textrm{ and } |d-a| > |b-a|  -1\\
M(\pi^{[a,b]} ,(\pi^{[c,d]})') & \textrm{ if } a\not\in [c,d] \textrm{ and } |d-a| \leq |b-a| -1 \\
M((\pi^{[a,b]}) ',(\pi^{[c,d] })') & \textrm{ if } a\not\in [c,d] \textrm{ and } |d-a| > |b-a|-1. \end{array} \right .  \]
If $a\in I_3$ and $b\not\in I_2$ then
\[ T(\pi^{[a,b]},\pi^{[c,d]}) = \left\{ \begin{array}{ll} M(\pi^{[a,b]},\pi^{[c,d]}) &  \textrm{ if } a-1\in [c,d] \textrm{ and } |d-(a-1)| \leq |b-a| +1 \\
M((\pi^{[a,b]})',\pi^{[c,d]}) & \textrm{ if }  a-1 \in [c,d] \textrm{ and } |d-(a-1)| > |b-a|+1  \\
M(\pi^{[a,b]} ,(\pi^{[c,d]})') & \textrm{ if } a-1\not\in [c,d] \textrm{ and } |d-(a-1)| \leq |b-a|+1  \\
M((\pi^{[a,b]}) ',(\pi^{[c,d] })') & \textrm{ if } a-1\not\in [c,d] \textrm{ and } |d-(a-1)| > |b-a|+1. \end{array} \right .  \]
If $a\in I_3$ and $b\in I_2$ then
\[ T(\pi^{[a,b]},\pi^{[c,d]}) = \left\{ \begin{array}{ll} M(\pi^{[a,b]},\pi^{[c,d]}) &  \textrm{ if } a-1\in [c,d] \textrm{ and } |d-(a-1)| \leq |b-a| \\
M((\pi^{[a,b]})',\pi^{[c,d]}) & \textrm{ if }  a-1 \in [c,d] \textrm{ and } |d-(a-1)| > |b-a| \\
M(\pi^{[a,b]} ,(\pi^{[c,d]})') & \textrm{ if } a-1\not\in [c,d] \textrm{ and } |d-(a-1)| \leq |b-a|  \\
M((\pi^{[a,b]}) ',(\pi^{[c,d] })') & \textrm{ if } a-1\not\in [c,d] \textrm{ and } |d-(a-1)| > |b-a|. \end{array} \right .  \]
\end{cor}
\begin{proof}We know that if $a\in I_1\cup I_3$ (resp. $a\in I_2$) then the map $L\in F(\pi^{[a,b]},\pi^{[c,d]})$ is determined by its value at the vectors $v^{a,b}_l(k)$ (resp. $Hv^{a,b}_l(k)$) for $1\leq k \leq b_l^{a,b}$. In the notation of the previous definition such a vector takes values in a space of dimension $d_{a,b;c,d}(l)$. So in all cases the total dimension of the space $F(\pi^{[a,b]},\pi^{[c,d]})$ equals
\[ T(\pi^{[a,b]},\pi^{[c,d]})= \sum_{l\geq 1} b_{l}^{a,b} \cdot d_{a,b;c,d}(l). \]
In the above definition of $d_{a,b;c,d}(l)$ there are four possible forms depending on the value of $a$ and $b$. Lets consider the first case when $a\in I_1\cup I_2$ and $b\not\in I_2$. Then we have that $d_{a,b;c,d}(l) = \dim\left(  V_a\cap V^{N\cdot (l-1) + |b-a| +1}\cap V^{c,d}  \right)$. Counting the number of basis vectors of $V^{c,d}$ that lie in $V_a$ we see there are four possibilities for $\dim \left( V_a\cap V^{c,d} \cap V^{N\cdot (l-1) +|b-a|+1} \right)$:
\[  \left. \begin{array}{ll} 
\sum_{i=1}^{l} ib_i^{c,d} + l\sum_{i\geq l} b_i^{c,d}&\textrm{ if } a\in [c,d] \textrm{ and } |d-a| \leq |b-a|  \\
\sum_{i=1}^{l-1} ib_i^{c,d} + (l-1)\sum_{i\geq l} b_i^{c,d} & \textrm{ if }  a \in [c,d] \textrm{ and } |d-a| > |b-a|  \\
\sum_{i=1}^{l} ib_{i+1}^{c,d} + l\sum_{i\geq l} b_{i+1}^{c,d} & \textrm{ if } a\not\in [c,d] \textrm{ and } |d-a| \leq |b-a|  \\
\sum_{i=1}^{l-1} ib_{i+1}^{c,d} + (l-1)\sum_{i\geq l} b_{i+1}^{c,d}  & \textrm{ if } a\not\in [c,d] \textrm{ and } |d-a| > |b-a|. 
\end{array} \right .\]
In the first case $ a\in [c,d] \textrm{ and } |d-a| \leq |b-a|$ and
\begin{eqnarray*}
T(\pi^{[a,b]},\pi^{[c,d]}) & = & \sum_{l\geq 1} b_l^{a,b} \cdot  \left( \sum_{i=1}^l ib_i^{c,d} + l\sum_{i\geq l} b_i^{c,d} \right) \\ 
& = & \sum_{i\geq 1} \left( \sum_{l\geq i} b^{a,b}_l \right) \cdot \left( \sum_{l\geq i} b^{c,d}_l \right) \\
 &=& M(\pi^{[a,b]},\pi^{[c,d]}).
\end{eqnarray*}
In the second case $a \in [c,d] \textrm{ and } |d-a| > |b-a|$ and
\begin{eqnarray*}
T(\pi^{[a,b]},\pi^{[c,d]}) & = & \sum_{l\geq 1} b_l^{a,b} \cdot  \left( \sum_{i=1}^{l-1} ib_i^{c,d} + (l-1)\sum_{i\geq l} b_i^{c,d} \right) \\ 
& = & \sum_{i\geq 1} \left( \sum_{l\geq i} b^{a,b}_{l+1} \right) \cdot \left( \sum_{l\geq i} b^{c,d}_l \right) \\
 &=& M((\pi^{[a,b]})',\pi^{[c,d]}).
\end{eqnarray*}
In the third case $a\not\in [c,d] \textrm{ and } |d-a| \leq |b-a|$ and
\begin{eqnarray*}
T(\pi^{[a,b]},\pi^{[c,d]}) & = & \sum_{l\geq 1} b_l^{a,b} \cdot  \left(\sum_{i=1}^{l} ib_{i+1}^{c,d} + l\sum_{i\geq l} b_{i+1}^{c,d}  \right) \\ 
& = & \sum_{i\geq 1} \left( \sum_{l\geq i} b^{a,b}_l \right) \cdot \left( \sum_{l\geq i} b^{c,d}_{l+1} \right) \\
 &=& M(\pi^{[a,b]},(\pi^{[c,d]})').
\end{eqnarray*}
Finally in the fourth case $a\not\in [c,d] \textrm{ and } |d-a| > |b-a|$ and we have
\begin{eqnarray*}
T(\pi^{[a,b]},\pi^{[c,d]}) & = & \sum_{l\geq 1} b_l^{a,b} \cdot  \left( \sum_{i=1}^{l-1} ib_{i+1}^{c,d} + (l-1)\sum_{i\geq l} b_{i+1}^{c,d} \right) \\ 
& = & \sum_{i\geq 1} \left( \sum_{l\geq i} b^{a,b}_{l+1} \right) \cdot \left( \sum_{l\geq i} b^{c,d}_{l+1} \right) \\
 &=& M((\pi^{[a,b]})',(\pi^{[c,d]})').
\end{eqnarray*}
This completes the situation when $a\in I_1\cup I_2$ and $b\not\in I_2$. In the other situations $a\in I_1\cup I_2$ and $b\in I_2$, or $a\in I_3$ and $b\not\in I_2$, or  $a\in I_3$ and $b\in I_2$. All of these cases can be dealt with in a similar manner.
\end{proof}
Now we have computed all the dimensions $T(\pi^{[a,b]},\pi^{[c,d]})$ and $B(\pi^{[a,b]},\pi^{[c,d]})$. The next lemma combines corollaries \ref{Bdim} and \ref{Tdim} to compute their difference. We see that in most cases there is an exact cancellation.
\begin{lem}\label{dif}We have
\[ T(\pi^{[a,b]},\pi^{[c,d]}) = B(\pi^{[a,b]},\pi^{[c,d]}) \textrm{ aside from the following cases, } \]
Case 1: $a\in I_1\cup I_2$, $b=d \in I_2$ 
\[ \begin{array}{ll} 
M((\pi^{[a,b]})',\pi^{[c,b]}) - M(\pi^{[a,b]},\pi^{[c,b]}) & \textrm{ if } a\in [c,b]  \\ 
M((\pi^{[a,b]})',(\pi^{[c,b]})') - M(\pi^{[a,b]},(\pi^{[c,b]})') & \textrm{ if } a\not\in [c,b].
\end{array}  \]
Case 2: $a\in I_3$, $b\not\in I_2$, $d=a-1\in I_2$
\[ \begin{array}{ll} 
M(\pi^{[a,b]},\pi^{[a,a-1]}) - M((\pi^{[a,b]})',\pi^{[a,a-1]}) & \textrm{ if } a=c  \\ 
M(\pi^{[a,b]},\pi^{[c,a-1]}) - M((\pi^{[a,b]})',(\pi^{[c,a-1]})') & \textrm{ if } a\not =c.
\end{array} \]
Case 3: $a\in I_3$, $b\not \in I_2$, $a=c$, $d\not =  a-1$,
\[ \begin{array}{ll} 
M(\pi^{[a,b]},(\pi^{[a,d]})') - M(\pi^{[a,b]},\pi^{[a,d]}) & \textrm{ if } |d-a| \leq |b-a|, \\
M((\pi^{[a,b]})',(\pi^{[a,d]})') - M((\pi^{[a,b]})',\pi^{[a,d]}) & \textrm{ if } |d-a| > |b-a|. 
 \end{array} \] 
Case 4: $a\in I_3$, $b\in I_2$, $d=a-1$,
\[ \begin{array}{ll}
M(\pi^{[a,b]},\pi^{[a,a-1]}) - M((\pi^{[a,b]})',\pi^{[a,a-1]}) & \textrm{ if } a=c \textrm{ and } b\not = a-1 \\
M(\pi^{[a,a-1]},\pi^{[c,a-1]}) - M(\pi^{[a,a-1]},(\pi^{[c,a-1]})') & \textrm{ if } a\not=c \textrm{ and } b=a-1 \\
M(\pi^{[a,b]},\pi^{[c,a-1]}) - M((\pi^{[a,b]})',(\pi^{[c,a-1]})') & \textrm{ if } a\not =c \textrm{ and } b\not=a-1 .
  \end{array} \]
Case 5: $a\in I_3$, $b\in I_2$, $a-1\in [c,d]$, $d\not =a-1$, $b=d$
\[ M((\pi^{[a,b]})',\pi^{[c,b]}) -M(\pi^{[a,b]},\pi^{[c,b]}). \]
Case 6: $a\in I_3$, $b\in I_2$, $a-1\not\in [c,d]$, $a=c$, $|d-a|<|b-a|$
\[ M(\pi^{[a,b]},(\pi^{[a,d]})')-M(\pi^{[a,b]},\pi^{[a,d]}). \]
Case 7: $a\in I_3$, $b\in I_2$, $a-1\not \in [c,d]$
\[ \begin{array}{ll}
M((\pi^{[a,b]})',(\pi^{[a,b]})') - M(\pi^{[a,b]},\pi^{[a,b]}) & \textrm{ if } a=c \textrm{ and } b=d \\
M((\pi^{[a,b]})',(\pi^{[a,d]})') - M((\pi^{[a,b]})',\pi^{[a,d]}) & \textrm{ if } a=c \textrm{ and } |d-a| >|b-a| \\
M((\pi^{[a,b]})',(\pi^{[c,b]})') - M(\pi^{[a,b]},(\pi^{[c,b]})') & \textrm{ if } a\not =c \textrm{ and } b=d .
\end{array} \]
\end{lem}
\begin{proof}Compare corollaries \ref{Bdim} and \ref{Tdim}.\end{proof}
Our aim throughout this appendix has been to prove Proposition \ref{difference} and deduce that the difference $\sum_{0\leq a,b,c,d \leq N-1}  T(\pi^{[a,b]},\pi^{[c,d]}) - B(\pi^{[a,b]},\pi^{[c,d]})$ equals 
\[-\frac{1}{2}\sum_{i\in I_2} \left( \sum_{ b \neq i } l(\pi^{[i+1,b]}) - \sum_{ c\neq i+1 }l(\pi^{[c,i]}) \right)^2 
- \half \sum_{a\in I_3 , b\not\in I_2} \sum_{i\geq 1}(b_i^{a,b})^2 - \half\sum_{a\not\in I_3 , b\in I_2}  \sum_{i\geq 1} (b_i^{a,b})^2. \]
So all that remains is to check this sum agrees with the values we computed. First we will transform it into a expression in terms of the $M(\pi^{[a,b]},\pi^{[c,d]})$. To do this we need the simple identities
\begin{eqnarray*} 
M(\pi^{[a,b]},\pi^{[c,d]}) - M((\pi^{[a,b]})',(\pi^{[c,d]})') &=& \sum_{l\geq 1}\left( \sum_{i\geq l} b^{a,b}_i \cdot \sum_{i\geq l} b^{c,d}_i  -   \sum_{i\geq l} b^{a,b}_{i+1} \cdot \sum_{i\geq l} b^{c,d}_{i+1} \right) \\
&=& \sum_{i\geq 1} b^{a,b}_i \cdot \sum_{i\geq 1} b^{c,d}_i\\
 & = &  l(\pi^{[a,b]})\cdot l(\pi^{[c,d]})
 \end{eqnarray*}
and
\begin{eqnarray*} 
M(\pi^{[a,b]},\pi^{[a,b]}) - M((\pi^{[a,b]})',\pi^{[a,b]}) &=& \sum_{l\geq 1} \left( \sum_{i\geq l} b^{a,b}_i \cdot \sum_{i\geq l} b^{a,b}_i  -   \sum_{i\geq l} b^{a,b}_{i+1} \cdot \sum_{i\geq l} b^{a,b}_{i} \right) \\
&=& \sum_{l\geq 1} b^{a,b}_l \cdot \sum_{i\geq l} b^{c,d}_i   \\
 & = &  \half l(\pi^{[a,b]})^2 + \half \sum_{l\geq 1} (b^{a,b}_l)^2.
 \end{eqnarray*}
 Using these two identities and some simple algebraic manipulations we can rewrite Proposition \ref{difference} as the statement that the difference $\sum_{0\leq a,b,c,d \leq N-1}  T(\pi^{[a,b]},\pi^{[c,d]}) - B(\pi^{[a,b]},\pi^{[c,d]})$ equals
\[ \begin{array}{llllll}
& & \sum_{i\in I_2} \sum_{b\not = i , c\not=i+1} & M(\pi^{[i+1,b]},\pi^{[c,i]}) &-& M((\pi^{[i+1,b]})',(\pi^{[c,i]})')  \\
& +&  \sum_{i\in I_2}  \sum_{\substack{b < d\\b,d\not = i }}& M((\pi^{[i+1,b]})',(\pi^{[i+1,d]})') &-& M(\pi^{[i+1,b]},\pi^{[i+1,d]}) \\
& +& \sum_{i\in I_2}  \sum_{\substack{a < c \\ a,c \not= i+1}} &M((\pi^{[a,i]})',(\pi^{[c,i]})') &-& M(\pi^{[a,i]},\pi^{[c,i]}) \\
& +& \sum_{a\in I_3 ,b\in I_2, b\not = a-1} &M((\pi^{[a,b]})',(\pi^{[a,b]})') &-&M(\pi^{[a,b]},\pi^{[a,b]}) \\
& +& \sum_{[a,b]\in S} &M((\pi^{[a,b]})',\pi^{[a,b]}) &-&M(\pi^{[a,b]},\pi^{[a,b]}).
 \end{array} \]
 We will take a systematic approach, accounting for these terms one by one, all in all we will check nine separate cases.
 
First let us assess the contribution from terms involving partitions $\pi^{[r,s]}$ with $r,s \in I_1$. Comparing with Lemma \ref{dif} in all seven cases there is no discrepancy when $a,b\in I_1$ or $c,d\in I_1$ and therefore there is no contribution from these terms in agreement with the above sum. 
 
Secondly we assess the contribution from terms involving partitions $\pi^{[r,s]}$ with $r\in I_1$ and $s\in I_2$. Considering Lemma \ref{dif} we note the following cases,
\[ \begin{array}{l}
\textrm{Case 1: } a\in I_1\cup I_2, b\in I_2 , b=d \\ 
\begin{array}{ll} 
M((\pi^{[a,b]})',\pi^{[c,b]}) - M(\pi^{[a,b]},\pi^{[c,b]}) & \textrm{ if } a\in [c,b]  \\ 
M((\pi^{[a,b]})',(\pi^{[c,b]})') - M(\pi^{[a,b]},(\pi^{[c,b]})') & \textrm{ if } a\not\in [c,b].
\end{array}  \\ \\
\textrm{Case 2: } a\in I_3, b\not \in I_2, c \in I_1, d=a-1\in I_2  \\ 
M(\pi^{[a,b]},\pi^{[c,a-1]}) - M((\pi^{a,b})',(\pi^{[c,a-1]})'). \\ \\
\textrm{Case 4: } a\in I_3, b\in I_2, c\in I_1, d=a-1\\ 
 \begin{array}{ll}
M(\pi^{[a,a-1]},\pi^{[c,a-1]}) - M(\pi^{[a,a-1]},(\pi^{[c,a-1]})') & \textrm{ if } b=a-1 \\
M(\pi^{[a,b]},\pi^{[c,a-1]}) - M((\pi^{[a,b]})',(\pi^{[c,a-1]})') & \textrm{ if } b\not=a-1 .
  \end{array} \\ \\
\textrm{Case 5: } a\in I_3, b\in I_2, c\in I_2, b=d, a-1\in [c,b], b\not = a-1 \\
M((\pi^{[a,b]})',\pi^{[c,b]}) - M(\pi^{[a,b]},\pi^{[c,b]}). \\ \\
\textrm{Case 7: } a \in I_3, b\in I_2, c \in I_1, b=d, a-1\not \in [c,b] \\
M((\pi^{[a,b]})' , (\pi^{[c,b]})') -  M(\pi^{[a,b]},(\pi^{[c,b]})') .
\end{array} \]
The sum total of these cases gives
\[ 
\begin{array}{ccccc}  
&\displaystyle \sum_{a\in I_1 , b\in I_2} &M((\pi^{[a,b]})',\pi^{[a,b]})&-& M(\pi^{[a,b]},\pi^{[a,b]}) \vspace{3pt}
\\
+&\displaystyle \sum_{\substack{a< c: a,c\not = b+1 \\ a\in I_1, b\in I_2 \textrm{ or } c\in I_1,b\in I_2 } }  &
M( (\pi^{[a,b]})' ,(\pi^{[c,b]})' ) & - & M( \pi^{[a,b]} ,\pi^{[c,b]} ) \vspace{3pt}
\\
+& \displaystyle\sum_{b\not = i  , c\in I_1,i \in I_2 } & M( \pi^{[i+1,b]} \pi^{[c,i]} ) & -& M(( \pi^{[i+1,b]})' ,(\pi^{[c,i]})' )
\end{array} \]
 this accounts for all the terms involving partitions $\pi^{[a,b]}$ with $a\in I_1$ and $b\in I_2$.
 
Thirdly we assess the contribution from terms involving partitions $\pi^{[r,s]}$ with $r\in I_1$ and $s\in I_3$. Comparing with lemma \ref{dif} in all seven cases there is no discrepancy when $a\in I_1$ and $b\in I_3$ or $c\in I_1$ and $d\in I_3$ and therefore there is no contribution from these terms in agreement with Proposition \ref{difference}. We have now observed the correct contributions from all terms involving partitions $\pi^{[r,s]}$ where $r\in I_1$.  

Fourthly we will consider contributions from terms involving partitions $\pi^{[r,s]}$ where $r\in I_2$ and $s\in I_1$. As in the first and third cases on comparing with lemma \ref{dif} in all seven cases there is no discrepancy when $a\in I_2$ and $b\in I_1$ or $c\in I_2$ and $d\in I_1$. Again this is in full agreement with Proposition \ref{difference}.

Fifthly we consider contributions from terms involving partitions $\pi^{[r,s]}$ where $r\in I_2$ and $s\in I_2$. This time on comparing with lemma \ref{dif} we observe some nontrivial contributions as desired. This case is almost identical to when $r\in I_1$ and $s\in I_2$. In the lemma the following cases contribute
\[ \begin{array}{l}
\textrm{Case 1: } a\in I_1\cup I_2, b\in I_2 , b=d \\ 
\begin{array}{ll} 
M((\pi^{[a,b]})',\pi^{[c,b]}) - M(\pi^{[a,b]},\pi^{[c,b]}) & \textrm{ if } a\in [c,b]  \\ 
M((\pi^{[a,b]})',(\pi^{[c,b]})') - M(\pi^{[a,b]},(\pi^{[c,b]})') & \textrm{ if } a\not\in [c,b].
\end{array}  \\ \\
\textrm{Case 2: } a\in I_3, b\not \in I_2, c \in I_2, d=a-1\in I_2  \\ 
M(\pi^{[a,b]},\pi^{[c,a-1]}) - M((\pi^{a,b})',(\pi^{[c,a-1]})'). \\ \\
\textrm{Case 4: } a\in I_3, b\in I_2, c\in I_2, d=a-1\\ 
 \begin{array}{ll}
M(\pi^{[a,a-1]},\pi^{[c,a-1]}) - M(\pi^{[a,a-1]},(\pi^{[c,a-1]})') & \textrm{ if } b=a-1 \\
M(\pi^{[a,b]},\pi^{[c,a-1]}) - M((\pi^{[a,b]})',(\pi^{[c,a-1]})') & \textrm{ if } b\not=a-1 .
  \end{array} \\ \\
\textrm{Case 5: } a\in I_3, b\in I_2, c\in I_2, b=d, a-1\in [c,b], b\not = a-1 \\
M((\pi^{[a,b]})',\pi^{[c,b]}) - M(\pi^{[a,b]},\pi^{[c,b]}). \\ \\
\textrm{Case 7: } a \in I_3, b\in I_2, c \in I_2, b=d, a-1\not \in [c,b] \\
M((\pi^{[a,b]})' , (\pi^{[c,b]})') -  M(\pi^{[a,b]},(\pi^{[c,b]})') .
\end{array} \]
The sum total of these cases gives
\[ \begin{array}{ccccc}
&\displaystyle\sum_{a\in I_2, b\in I_2} & M((\pi^{[a,b]})',\pi^{[a,b]}) & - & M(\pi^{[a,b]} ,\pi^{[a,b]}) \vspace{3pt}\\
+&\displaystyle\sum_{\substack{a<c: a,c\not = b+1 \\ a ,b\in I_2 \textrm{ or } c, b\in I_2 } } &
M((\pi^{[a,b]})',(\pi^{[c,b]})') & - & M(\pi^{[a,b]},\pi^{[c,b]}) \vspace{3pt}\\ 
+&\displaystyle\sum_{ c \in I_2 i \in I_2 b\not = i} & M( \pi^{[i+1,b]} \pi^{[c,i]} ) & - & M(( \pi^{[i+1,b]})' ,(\pi^{[c,i]})' ).
\end{array} \]
These terms again agree with those of Proposition \ref{dif}. 

Sixthly we move to consider terms involving partitions $\pi^{[r,s]}$ where $r\in I_2$ and $s\in I_3$. considering the lemma we see that there is no contribution for these terms as desired.  We have now observed the correct contributions from all terms involving partitions $\pi^{[r,s]}$ where $r\in I_1\cup I_2$ only the cases when $r\in I_3$ remain, we may restrict to consider only those differences $T(\pi^{[a,b]},\pi^{[c,d]}) -B(\pi^{[a,b]},\pi^{[c,d]}) $ where both $a,c\in I_3$.

Seventhly we consider terms involving partitions $\pi^{[r,s]}$ where $r\in I_3$ and $s\in I_1$. Considering lemma \ref{dif} we see that the following cases give non-trivial contributions.
\[ \begin{array}{l}
\textrm{Case 2: } a\in I_3, b\not \in I_2, d=a-1\in I_2  \\ 
\begin{array}{ll} 
M(\pi^{[a,b]},\pi^{[a,a-1]}) - M((\pi^{[a,b]})',\pi^{[a,a-1]}) & \textrm{ if } a=c  \\ 
M(\pi^{[a,b]},\pi^{[c,a-1]}) - M((\pi^{[a,b]})',(\pi^{[c,a-1]})') & \textrm{ if } a\not =c.
\end{array} \\ \\
\textrm{Case 3: }  a\in I_3, b\not \in I_2, a=c, d\not= a-1  \\
\begin{array}{ll} 
M(\pi^{[a,b]},(\pi^{[a,d]})') - M(\pi^{[a,b]},\pi^{[a,d]}) & \textrm{ if } |d-a| \leq |b-a|, \\
M((\pi^{[a,b]})',(\pi^{[a,d]})') - M((\pi^{[a,b]})',\pi^{[a,d]}) & \textrm{ if } |d-a| > |b-a|. 
 \end{array} \\ \\
\textrm{Case 6: } a\in I_3, b\in I_2, a-1\not\in [c,d], a=c, |d-a|<|b-a|, d \in I_1 \\
 M(\pi^{[a,b]},(\pi^{[a,d]})')-M(\pi^{[a,b]},\pi^{[a,d]}). \\ \\
\textrm{Case 7: } a\in I_3, b\in I_2, a-1\not\in [c,d], d \in I_1 \\
\begin{array}{ll}
M((\pi^{[a,b]})',(\pi^{[a,d]})') - M((\pi^{[a,b]})',\pi^{[a,d]}) & \textrm{ if } a=c \textrm{ and } |d-a| >|b-a|. \\ \\
\end{array}
\end{array} \]
The sum total of these cases gives
\[ \begin{array}{ccccc}
& \displaystyle\sum_{a\in I_3, b\in I_1} & M((\pi^{[a,b]})',\pi^{[a,b]}) & - & M(\pi^{[a,b]} ,\pi^{[a,b]}) \vspace{3pt}\\
+& \displaystyle\sum_{\substack{b<d: b,d\not = a \\ a\in I_2 \textrm{ and } b\in I_1 \textrm{ or } d\in I_1 } } &
M((\pi^{[a+1,b]})',(\pi^{[a+1,d]})') & - & M(\pi^{[a+1,b]},\pi^{[a+1,d]}) \vspace{3pt}\\ 
+& \displaystyle\sum_{i \in I_2 b\in I_1} & M( \pi^{[i+1,b]} \pi^{[c,i]} ) & - & M(( \pi^{[i+1,b]})' ,(\pi^{[c,i]})' ).
\end{array} \]
Eighthly we consider terms involving partitions $\pi^{[r,s]}$ where $r\in I_3$ and $s\in I_3$. Here considering lemma \ref{dif} we see that the following cases give non-trivial contributions.
\[ \begin{array}{l}
\textrm{Case 2: } a\in I_3, b\not \in I_2, d=a-1\in I_2  \\ 
\begin{array}{ll} 
M(\pi^{[a,b]},\pi^{[a,a-1]}) - M((\pi^{[a,b]})',\pi^{[a,a-1]}) & \textrm{ if } a=c  \\ 
M(\pi^{[a,b]},\pi^{[c,a-1]}) - M((\pi^{[a,b]})',(\pi^{[c,a-1]})') & \textrm{ if } a\not =c.
\end{array} \\ \\
\textrm{Case 3: }  a\in I_3, b\not \in I_2, a=c, d\not= a-1  \\
\begin{array}{ll} 
M(\pi^{[a,b]},(\pi^{[a,d]})') - M(\pi^{[a,b]},\pi^{[a,d]}) & \textrm{ if } |d-a| \leq |b-a|, \\
M((\pi^{[a,b]})',(\pi^{[a,d]})') - M((\pi^{[a,b]})',\pi^{[a,d]}) & \textrm{ if } |d-a| > |b-a|. 
 \end{array} \\ \\
\textrm{Case 6: } a\in I_3, b\in I_2, a-1\not\in [c,d], a=c, |d-a|<|b-a|, d \in I_3 \\
 M(\pi^{[a,b]},(\pi^{[a,d]})')-M(\pi^{[a,b]},\pi^{[a,d]}). \\ \\
\textrm{Case 7: } a\in I_3, b\in I_2, a-1\not\in [c,d], d \in I_3 \\
\begin{array}{ll}
M((\pi^{[a,b]})',(\pi^{[a,d]})') - M((\pi^{[a,b]})',\pi^{[a,d]}) & \textrm{ if } a=c \textrm{ and } |d-a| >|b-a|. \\ \\
\end{array}
\end{array} \]
The sum total of these cases gives
\[ \begin{array}{ccccc}
&\displaystyle\sum_{a\in I_3, b\in I_3} & M((\pi^{[a,b]})',\pi^{[a,b]}) & - & M(\pi^{[a,b]} ,\pi^{[a,b]}) \vspace{3pt}\\
+&\displaystyle\sum_{\substack{b<d: b,d\not = a \\ a\in I_2 \textrm{ and } b\in I_3 \textrm{ or } d\in I_3 } } &
M((\pi^{[a+1,b]})',(\pi^{[a+1,d]})') & - & M(\pi^{[a+1,b]},\pi^{[a+1,d]})\vspace{3pt} \\ 
+&\displaystyle\sum_{i \in I_2 b\in I_3} & M( \pi^{[i+1,b]} \pi^{[c,i]} ) & - & M(( \pi^{[i+1,b]})' ,(\pi^{[c,i]})' ).
\end{array} \]
Now we have accounted for all terms involving partitions other that the case $\pi^{[r,s]}$ with $r\in I_3$ and $s\in I_2$. So we can now restrict to consider terms $T(\pi^{[a,b]},\pi^{[c,d]}) - B(\pi^{[a,b]},\pi^{[c,d]})$ with $a,c\in I_3$ and $b,d\in I_2$ as follows.

Ninthly we consider terms involving partitions $\pi^{[r,s]}$ where $r\in I_3$ and $s\in I_2$. Here considering lemma \ref{dif} we see that the following cases give non-trivial contributions.
\[ \begin{array}{l}
 \textrm{Case 4: } a\in I_3, b\in I_2, d=a-1, \\
\begin{array}{ll}
M(\pi^{[a,b]},\pi^{[a,a-1]}) - M((\pi^{[a,b]})',\pi^{[a,a-1]}) & \textrm{ if } a=c \textrm{ and } b\not = a-1 \\
M(\pi^{[a,a-1]},\pi^{[c,a-1]}) - M(\pi^{[a,a-1]},(\pi^{[c,a-1]})') & \textrm{ if } a\not=c \textrm{ and } b=a-1 \\
M(\pi^{[a,b]},\pi^{[c,a-1]}) - M((\pi^{[a,b]})',(\pi^{[c,a-1]})') & \textrm{ if } a\not =c \textrm{ and } b\not=a-1 .
  \end{array} \\ \\
\textrm{Case 5: }a\in I_3, b\in I_2, a-1\in [c,d], d\not =a-1, b=d \\
 M((\pi^{[a,b]})',\pi^{[c,b]}) -M(\pi^{[a,b]},\pi^{[c,b]}). \\ \\
\textrm{Case 6: } a\in I_3, b\in I_2, a-1\not\in [c,d], a=c, |d-a|<|b-a| \\
 M(\pi^{[a,b]},(\pi^{[a,d]})')-M(\pi^{[a,b]},\pi^{[a,d]}). \\ \\
\textrm{Case 7: } a\in I_3, b\in I_2, a-1\not\in [c,d]  \\
 \begin{array}{ll}
M((\pi^{[a,b]})',(\pi^{[a,b]})') - M(\pi^{[a,b]},\pi^{[a,b]}) & \textrm{ if } a=c \textrm{ and } b=d \\
M((\pi^{[a,b]})',(\pi^{[a,d]})') - M((\pi^{[a,b]})',\pi^{[a,d]}) & \textrm{ if } a=c \textrm{ and } |d-a| >|b-a| \\
M((\pi^{[a,b]})',(\pi^{[c,b]})') - M(\pi^{[a,b]},(\pi^{[c,b]})') & \textrm{ if } a\not =c \textrm{ and } b=d .
\end{array}
\end{array} \]
The sum total of these cases gives the remaining terms
\[\begin{array}{ccccc}
&\displaystyle\sum_{a\in I_3 , b\in I_2, a-1\not= b} & M((\pi^{[a,b]})',(\pi^{[a,b]})') &-& M(\pi^{[a,b]},\pi^{[a,b]}) \vspace{3pt}\\
+&\displaystyle\sum_{\substack{ b<d : b,d \neq a \\ a\in I_2 \textrm{ and } b\in I_2  \textrm{ or } d\in I_2 }} 
& M((\pi^{[a+1,b]})',(\pi^{[a+1,d]})') &- &M(\pi^{[a+1,b]},\pi^{[a+1,d]}) \vspace{3pt}\\
+&\displaystyle\sum_{\substack{ a<c: a,c\neq b \\ b\in I_2 \textrm{ and } a\in I_3 \textrm{ or } c\in I_3}} 
& M((\pi^{[a,b]})',(\pi^{[c,b]})') &- & M(\pi^{[a,b]},\pi^{[c,b]}) \vspace{3pt}\\
+&\displaystyle\sum_{\substack{b\not = i \textrm{ and } a-1 \not = i \\ i \in I_2 \textrm{ and } b\in I_2 \textrm{ or } c\in I_3 }}
& M(\pi^{[i+1,b]},\pi^{[c,i]}) &- & M((\pi^{[i+1,b]})',(\pi^{[c,i]})') .
\end{array}\]
This completes the proof of Proposition \ref{difference}.